\documentclass[11pt,a4paper]{article}
\usepackage[margin=1.206in]{geometry}
\usepackage[latin1]{inputenc}
\usepackage[boxed]{algorithm2e}
\usepackage{amsmath}
\usepackage{amsthm}
\usepackage{amssymb,url}
\usepackage{amsfonts}
\usepackage{amscd}%
\usepackage{graphicx}
\usepackage{dsfont}
\usepackage{multirow}
\usepackage{enumerate}
\usepackage{scrpage2}
\usepackage{float}
\usepackage{libertine}
\usepackage[libertine]{newtxmath}
\usepackage{bm}
\usepackage{caption}
\usepackage{subcaption}
\usepackage[dvipsnames]{xcolor}
\usepackage[citecolor=NavyBlue,colorlinks=true,linkcolor=NavyBlue,urlcolor=NavyBlue]{hyperref}
\usepackage[compact,raggedright,small]{titlesec}
\usepackage{enumitem}
\usepackage{cleveref}
\usepackage{lineno}
\modulolinenumbers[5]
\DeclareMathOperator{\sign}{sign}
\DeclareMathOperator{\e}{\mathrm{e}}

\DeclareMathOperator{\supp}{supp}
\DeclareMathOperator*{\argmin}{arg\,min}

\DeclareMathOperator{\R}{\mathbb{R}}

\DeclareMathOperator{\N}{\mathbb{N}}

\DeclareMathOperator{\E}{\mathbb{E}}
\DeclareMathOperator{\Pro}{\mathbb{P}}

\newcommand{\triplenorm}{{|\!|\!|}}

\newtheorem{theorem}{Theorem}
\newtheorem{lemma}[theorem]{Lemma}

\newtheorem{corollary}[theorem]{Corollary}
\newtheorem{remark}[theorem]{Remark}
\newtheorem{example}[theorem]{Example}
\newtheorem{definition}[theorem]{Definition}
\crefformat{equation}{\textup{#2(#1)#3}}
\crefrangeformat{equation}{\textup{#3(#1)#4--#5(#2)#6}}
\crefmultiformat{equation}{\textup{#2(#1)#3}}{ and \textup{#2(#1)#3}}
{, \textup{#2(#1)#3}}{, and \textup{#2(#1)#3}}
\crefrangemultiformat{equation}{\textup{#3(#1)#4--#5(#2)#6}}%
{ and \textup{#3(#1)#4--#5(#2)#6}}{, \textup{#3(#1)#4--#5(#2)#6}}{, and \textup{#3(#1)#4--#5(#2)#6}}

\Crefformat{equation}{#2Equation~\textup{(#1)}#3}
\Crefrangeformat{equation}{Equations~\textup{#3(#1)#4--#5(#2)#6}}
\Crefmultiformat{equation}{Equations~\textup{#2(#1)#3}}{ and \textup{#2(#1)#3}}
{, \textup{#2(#1)#3}}{, and \textup{#2(#1)#3}}
\Crefrangemultiformat{equation}{Equations~\textup{#3(#1)#4--#5(#2)#6}}%
{ and \textup{#3(#1)#4--#5(#2)#6}}{, \textup{#3(#1)#4--#5(#2)#6}}{, and \textup{#3(#1)#4--#5(#2)#6}}

\makeatletter
\newcommand{\footremember}[2]{%
\footnote{#2}
\newcounter{#1}
\setcounter{#1}{\value{footnote}}%
}
\newcommand{\footrecall}[1]{%
\footnotemark[\value{#1}]%
}
\makeatother
\title{\large\bfseries A Quotient Property for Matrices with Heavy-Tailed Entries and its Application to Noise-Blind Compressed Sensing}
\date{\today}
\author{Felix Krahmer\footremember{TUM}{Department of Mathematics, Technische Universit\"at M\"unchen, 85748 Garching bei M\"unchen, Germany \newline (\href{mailto:felix.krahmer@tum.de}{felix.krahmer@tum.de}, \href{mailto:christian.kuemmerle@ma.tum.de}{christian.kuemmerle@ma.tum.de})} \and Christian K\"ummerle\footrecall{TUM} \and
Holger Rauhut\footremember{RWTH}{Chair for Mathematics C (Analysis), RWTH Aachen University, 52062 Aachen, Germany \newline (\href{rauhut@mathc.rwth-aachen.de}{rauhut@mathc.rwth-aachen.de})}}
\begin{document}

\maketitle
\vspace{-3mm}
{\let\thefootnote\relax\footnote{{
\hspace{-6.4mm}2010 Mathematics Subject Classification: 46B20, 46B09, 15A52, 65K10, 52A22. 

\hspace{-3.9mm} Key words: random matrices, quotient property, random polytopes, compressive sensing, $\ell_1$-minimization.}}}

\begin{abstract}
%v3: \\
For a large class of random matrices $A$ with i.i.d. entries we show that the $\ell_1$-quotient property holds with probability exponentially close to 1. In contrast to previous results, our analysis does not require concentration of the entrywise distributions. We provide a unified proof that recovers corresponding previous results for (sub-)Gaussian and Weibull distributions. Our findings generalize known results on the geometry of random polytopes, providing lower bounds on the size of the largest Euclidean ball contained in the centrally symmetric polytope spanned by the columns of $A$. \\
At the same time, our results establish robustness of noise-blind $\ell_1$-decoders for recovering sparse vectors $x$ from underdetermined, noisy linear measurements $y=Ax+w$ under the weakest possible assumptions on the entrywise distributions that allow for recovery with optimal sample complexity even in the noiseless case. 
Our analysis predicts superior robustness behavior for measurement matrices with super-Gaussian entries, which we confirm by numerical experiments.
\end{abstract}
\vspace{4mm}
\section{Introduction}
\subsection{Random polytopes}
%\subsection{Notes for random polytopes}
%Let $K= \{b_1,...b_N \} \subset \R^m$ be a finite set. Then we define its \emph{(convex) polytope} $P_K$ as
%\[
%P_K = \operatorname{conv}(K) = \{u_1\}
%\]
%\subsection{Intro}
Let $A = (a_{ji})$ be a rectangular $m \times N$ random matrix with independent, symmetric and unit variance entries $a_{ji}$ and $m < N$, and denote by $B_p^N$ the unit ball of the $\ell_p$-norm in $\R^N$. In this paper, we study the geometry of the image $A B_1^N$ under quite general assumptions on the distribution of the entries $a_{ji}$. This object can be also regarded as the random polytope defined by the absolute convex hull of the columns of A, i.e.,
\[
A B_1^N =  \operatorname{span}\{\pm a_{1},\ldots,\pm a_{N}\},
\]
if $a_1,\ldots,a_N$ denote the columns of $A$.

For normally distributed $a_{ji}$, a result due to Gluskin and Kashin quantifies  the inclusion of an Euclidean ball $B_2^N$ in $A B_1^N$.
\begin{theorem}[\cite{Gluskin89,Kashin83,Wojtaszczyk10}] \label{thm_Gluskin}
If the $a_{ji}$ are independent mean-zero, variance one Gaussian random variables, there exist constants $C,D > 0$ such that if $N \geq 2 m$,
\begin{equation} \label{eq_result_Gluskin}
\Pro \Big( A B_1^N \supset  \frac{\sqrt{\log{(eN/m)}}}{D} B_2^N \Big) \geq 1 - \exp(-C m).
\end{equation}
\end{theorem}
This statement corresponds to a lower bound on the \emph{inradius} of $A B_1^N$, i.e., the radius of the largest Euclidean ball that is contained in the random polytope $A B_1^N$.

Litvak et al. proved a similar result for $a_{ji}$ that fulfill a concentration property. Below $\|A\|_{2 \to 2}$ denotes
the spectral norm of a matrix $A$.
\begin{theorem}[{\cite{Litvak05}, see also \cite[Theorem 11.21.]{FoucartRauhut13}}] \label{thm_Litvak}
If there exist constants $a_1, a_2 > 0$ such that $\Pro( \|A\|_{2 \to 2} \geq a_1 \sqrt{N}) \leq \exp(-a_2 N)$ and if the third moments of the $a_{ji}$ are bounded by $\mu$, there exist constants $D(\mu), \widetilde{C}(a_1), C(a_2) > 0$ and $c > 0$ such that if $N \geq \widetilde{C}(a_1) m$,
\begin{equation} \label{eq_result_Litvak}
\Pro \Big( A B_1^N \supset  \frac{1}{D(\mu)} \big(\sqrt{\log{(eN/m)}} B_2^N \cap B_\infty^N \big) \Big) \geq 1 - c \exp(-C(a_2) m).
\end{equation}
\end{theorem}
An important instance of distributions fulfilling the assumption of the latter result are sub-Gaussian distributions. By considering symmetric $\pm 1$ random variables $a_{ji}$ (which are sub-Gaussian), it can be seen that the intersection of $\sqrt{\log{(eN/m)}} B_2^N$ with the unit cube $B_\infty^N$ in the last result is indeed necessary \cite{DeVoreWojtaszczyk09} for general sub-Gaussian distributions. In follow-up works, corresponding results have also been obtained for matrices $A$ with dependent entries, most notably the scenario where the vertices of $A B_1^N$ are drawn uniformly from a convex body \cite{Dafnis09}, \cite[Chapter 11]{GeoIsoConvex14}. Due to a close connection to log-concave measures, this model can also be seen as a version of a concentration requirement (however, 
weaker than subgaussianity).

In this paper, we establish results corresponding to \Cref{thm_Gluskin} and \Cref{thm_Litvak} for a significantly enlarged class of random matrices $A$. In particular, this class includes \emph{heavy-tailed} entry-wise distributions which do not fulfill strong concentration properties.
Following the arguments in \cite{Litvak05}, the resulting lower bounds on the inradius have implications on bounds of other geometric quantities of the corresponding random polytopes such as their volume and their mean width. 
These lower bounds on the volume of random polytopes have also been used in the context of differential privacy \cite{HardtTalwar10}.

\subsection{The quotient property in compressive sensing}
Our analysis is additionally motivated by the theory of compressive sensing, which studies the recovery of sparse vectors from incomplete linear measurements via efficient methods such as $\ell_1$-minimization \cite{FoucartRauhut13}. 
Provably optimal guarantees are available for random matrices. While previous work 
has mostly considered random matrices with entries obeying strong concentration properties such as Gaussian and subgaussian random random variables, it has recently been shown that concentration is not required for sparse
recovery guarantees. 
More precisely, Lecu{\'e} and Mendelson \cite{ML17}, see also \cite{DLR16}, showed recovery results for random matrices with 
independent, possibly heavy-tailed entries, requiring only $\log(N)$ finite moments. 
Their proof establishes the null space property via Mendelson's small ball method \cite{MendelsonLearning15,KoltchinskiiMendelson15}. 

Our work extends this line of research and yields recovery guarantees for unknown noise levels without requiring concentration on the entries of the measurement matrix. 
As observed in \cite{Wojtaszczyk10,DeVoreWojtaszczyk09}, this problem is closely connected to statements about polytope inclusions as given in \cref{eq_result_Gluskin,eq_result_Litvak}, respectively, which in this context are commonly referred to as quotient properties, see \Cref{def_QP} below for a precise definition. More precisely, our results imply stable and 
robust recovery for equality-constrained 
$\ell_1$-minimization from noisy, random measurements with heavy-tailed matrix entries without requiring an a-priori estimate of the noise level as would be needed for standard noise-aware $\ell_1$-minimization (basis pursuit denoising). 

Stated formally, we seek to recover a vector $x \in \R^N$ from noisy, underdetermined measurements
\[
y = Ax + w,
\]
where $A \in \R^{m \times N}$ with $m < N$ is the so-called measurement matrix and $w \in \R^m$ is a noise vector.
If $x$ is $s$-sparse, i.e., $\|x\|_0 := \#\{j : x_j \neq 0\}$, or approximately $s$-sparse in the sense that
\[
\sigma_s(x)_1 := \inf \{\|x-z\|_1: z\in \R^N, \|z\|_0 \leq s\}
\]
is small,
then we can hope to do so via $\ell_1$-minimization
\begin{equation}\label{l1:min}
\min_{z \in \R^N} \|z\|_1 \quad \mbox{ subject to } A z = y.
\end{equation}
In fact, if $A$ is an $m \times N$ matrix with independent standard Gaussian $\mathcal{N}(0,1)$ random variables,
\[
m \geq Cs\log(eN/s)
\]
and $y = Ax$,
then with high probability the minimizer $x^\sharp$ of \cref{l1:min} coincides with $x$ if $\|x\|_0 \leq s$
and more generally \cite{FoucartRauhut13},
\[
\| x - x^\sharp\|_1 \leq C \sigma_1(x)_s\quad \mbox{and} \quad  \| x - x^\sharp\|_2 \leq C \frac{\sigma_1(x)_s}{\sqrt{s}}.
\] 
In the noisy case $y = Ax + w$ with known noise bound $\|w\|_2 \leq \eta$, one commonly considers the constrained $\ell_1$-minimization problem
\begin{equation}\label{l1:noise}
\min_{z \in \R^N} \|z\|_1 \quad \mbox{ subject to } \quad \|Az - y \|_2 \leq \eta.
\end{equation}
For a Gaussian $m \times N$ matrix with $m \geq C s \log(eN/s)$, the minimizer $x^\sharp$ of \cref{l1:noise} satisfies
\begin{equation}\label{rec:bound:noisy}
\| x - x^\sharp\|_1 \leq C \sigma_1(x)_s + D \sqrt{s} \eta \quad \mbox{ and } \quad \|x-x^\sharp\|_2 \leq C 
\frac{\sigma_1(x)_s}{\sqrt{s}} + D \eta. 
\end{equation}
In practice, however, an accurate noise bound $\eta$ may not be known. If $\eta$ is an underestimation of the true $\|w\|_2$
then the so-called restricted isometry property or the robust null space property as used in the standard proofs \cite[Chapters 4, 6]{FoucartRauhut13} are not sufficient to guarantee the error bounds \cref{rec:bound:noisy}. If $\eta$ is
an overestimation of $\|w\|_2$ then the bounds \cref{rec:bound:noisy} may be very pessimistic as they depend on $\eta$
rather on the true noise level $\|w\|_2$ (see also Chapter~\ref{sec:numerical} for corresponding numerical experiments).

In order to address this problem, Wojtaszczyk suggested to simply 
use equality-constrained $\ell_1$-minimization \cite{Wojtaszczyk10} and provided an 
analysis for Gaussian measurement matrices $A$ based on \Cref{thm_Gluskin}, which was later adapted to subgaussian matrices \cite{DeVoreWojtaszczyk09} using \Cref{thm_Litvak} and also to Weibull matrices \cite{Foucart14}.
%quotient-property of the measurement matrix, see Definition~\ref{def_QP} below. 
%Gaussian random matrices $A$ satisfy this property 
%(with respect to the $\ell_2$-norm) with high probability
%if $m \sim s \log(eN/s)$ so that approximately sparse 
%vectors $x$ can approximately be recovered from $y = Ax + w$ via \cref{l1:min}, 
%see also \cite[Chapter 11]{FoucartRauhut13}.
The resulting error bound is of the form
\[  
\| x - x^\sharp\|_1 \leq C \sigma_1(x)_s + D \sqrt{s} \triplenorm w\triplenorm \quad \mbox{ and } \quad \|x-x^\sharp\|_2 \leq C 
\frac{\sigma_1(x)_s}{\sqrt{s}} + D \triplenorm w\triplenorm,
\]
where $\triplenorm w\triplenorm$ is the Euclidean norm for Gaussian and Weibull matrices and an interpolation norm between Euclidean and supremum norm for subgaussian matrices.
%Subgaussian matrices such as Bernoulli matrices unfortunately do not satisfy the quotient property with respect to $\ell_2$
%in general. However, by slightly changing the norm in which noise is measured this problem can be mended, see also below. 

Similar results have been recently established in \cite{BA18} for constrained $\ell_1$-minimization \cref{l1:noise}, where
$\eta$ is possibly underestimated, i.e., $\|w\|_2 \geq \eta$, see also below.

\subsection{Outline and contribution of this paper}
Our main contribution is twofold: Firstly, we prove a significantly generalized version of \Cref{thm_Gluskin} and the Gluskin-type inclusion \cref{eq_result_Gluskin} as compared to the ones for Gaussian \cite{Gluskin89} or Weibull \cite{Foucart14} distributions; namely, our result only requires (independent) matrix entries to be super-Gaussian (for the precise meaning of this concept, we refer to \Cref{def_SuperGaussian}), see \Cref{theorem_QP}(b) and  \Cref{cor_QP}(b). Secondly, in \Cref{theorem_QP}(a) and \Cref{cor_QP}(a), we generalize \Cref{thm_Litvak}, requiring only entrywise distributions with logarithmically many well-behaved moments. In both parts of \Cref{theorem_QP}, our results are expressed in terms of the \emph{$\ell_1$-quotient property}. All these concepts and results are introduced in detail in \Cref{subsec_QP}.

Based on \Cref{theorem_QP}, we provide, in \Cref{subsec_RobustnessGuarantees}, new robustness guarantees for noise-blind $\ell_1$-minimization for measurement matrices with quite general entrywise distributions in the regime of optimal sample complexity $m \approx C s \log(\e N /s)$ in \Cref{theorem_robustnessGuarantees}. The requirements on the entrywise distributions match the relatively weak moment assumptions of \cite{ML17} that can be shown to be almost necessary in the regime of optimal sample complexity for sparse recovery even in the noiseless case. Our result covers both the case of equality-constrained $\ell_1$-minimization (cf.~\Cref{remark_7}) and the case of quadratically constrained $\ell_1$-minimization with underestimated noise level, as studied in \cite{BA18}.

Notably, we provide a unified proof strategy for our results, which covers all previous results for matrices with independent entries, both on Gluskin-type inclusions and on the robustness of noise-blind $\ell_1$-minimization.
The proofs of our results can be found in \Cref{sec_proofs}.

In \Cref{sec:numerical}, our results are complemented by numerical experiments, confirming the robustness of noise-blind $\ell_1$-minimization for certain heavy-tailed measurement scenarios and exploring the recovery properties for different types of noise.

\subsection{Notation}
In this section, we recall some of the notation we use in this paper. 
For $N \in \N$, we write $[N] := \{k \in \N : 1 \leq k \leq N\}$.
For a vector $x \in \R^N$, 
we write $\|x\|_p = \big( \sum_{j=1}^N |x_j|^p\big)^{1/p}$, $1 \leq p < \infty$,
%for its $\ell_p$-norm 
and $\|x\|_\infty = \sup_{j \in [N]} |x_j|$ for its $\ell_p$-norm, while for
a random variable $X$ taking values a normed vector space, 
we denote by $\|X\|_p = \big(\mathbb{E}[\|X\|^p]\big)^{1/p}$, $1 \leq p < \infty$, its \emph{$p$-th moment}.
For $1 \leq p \leq \infty$, we denote the unit ball of the $\ell_p$-ball in $\R^N$ as $B_p^N = \{ x \in \R^N : \|x\|_p \leq 1\}$.
The clipped $\ell_2$-norm with parameter $\alpha$ is defined as $\|\cdot\|^{(\alpha)}:= \max\{\|\cdot \|_2,\alpha \|\cdot\|_\infty\}$.
A Rademacher sequence $\epsilon = (\epsilon_i)_{i \in [N]}$ is a sequence of independent random variables $\epsilon_i$ taking the values $-1$ and $+1$ with equal probability.

\section{Main results} \label{sec_main_results}%Let $A = (a_{ji})$ be an $m \times N$ matrix. 
We first state the results about the quotient property and its implication for the geometry of the polytope spanned by the columns of a random matrix. We distinguish two types of assumptions on the entrywise distributions its entries.
\begin{definition}\label{def_SuperGaussian} 
Let $X$ be a random variable with $\E[X] = 0$ and unit variance (so that $\|X\|_{L_2} =1$).
\begin{enumerate}
\item $X$ is called a \emph{super-Gaussian variable with parameter $\sigma > 0 $} if there exists $\sigma > 0$ such that
\begin{equation} \label{eq_assumption_supergaussian_lemma}
\Pro( |\sigma g| > t) \leq \Pro( | X | > t)
\end{equation}
for all $t > 0$, where $g$ is a standard normal random variable.
\item $X$ is said to fulfill the \emph{weak moment assumption of order $k$ with constants $\kappa_1$ and $\gamma \geq 1/2$} if
\[
\|X\|_{L_p} \leq \kappa_1 p^{\gamma} \quad \text{ for all }4 \leq p \leq k.
\]
\end{enumerate}
\end{definition}

\subsection{Quotient properties and polytope geometry} \label{subsec_QP}
The \emph{$\ell_1$-quotient property} as given in the following definition is a main object of our studies.
\begin{definition}[\cite{Foucart14,Wojtaszczyk10}] \label{def_QP}
A matrix $A \in \R^{m \times N}$ is said to possess the \emph{$\ell_1$-quotient property} with constant $d$ relative to a norm $\|\cdot\|$ on $\R^m$ if, for all $w \in \R^m$, there exists $u \in \R^N$ such that 
\[
A u = w \quad \text{ and }\quad \|u\|_{1} \leq d s_*^{1/2} \|w\|,
\]
with $s_* = m / \log(\e N/m)$.
\end{definition}

We proceed to our main theoretical result. We note that the assumption of identical distributions can be relaxed, but for simplicity we present the theorem under this assumption.

\begin{theorem} \label{theorem_QP}
Let $A = (a_{ji})$ %_{1 \leq j \leq m}^{1 \leq i \leq N} \in \R^{m \times N}$ 
be an $m \times N$ random matrix with independent symmetric, unit variance entries $a_{ji} \sim X$ for all $j \in [m]$, $i \in [N]$. %and $A := \frac{1}{\sqrt{m}}B$. 

%Assume that $X$ is a \emph{symmetric} random variable such that $\|X\|_{L_2} =1$ and there are constants $\gamma \geq 1/2$ and $\kappa_1 >1$ such that
%\begin{equation} \label{eq_moment_assumption}
%\|X\|_{L_p} \leq \kappa_1 p^\gamma \text{ for all }4 \leq p \leq \max\big(4,\log(m)\big).
%\end{equation}
%%%%%%there exist constants $\kappa_1, \kappa_2, w > 1$ and an $\gamma$-dependent constant $c_3 (\gamma) > 0$ such that
%%%%%%\[
%%%%%%\|X\|_{L_p} \leq \kappa_1 p^\gamma \text{ for all }4 \leq p \leq 2 \kappa_2 \log(wm).
%%%%%%\]
%Then there exist constants $\widetilde{C} > 0 $ and $D > 0$  that depend only on $\kappa_1$ and $\gamma$\footnote{e.g., $\widetilde{C}:= \max\big\{64^4 e,C\big\}$ with $C:=\frac{\kappa_1^{16} 4^{16 \gamma+6}}{\e}$) and that $m$ is large enough to fulfill $4 \e^{4 \gamma} c_0^2 \kappa_1 \log(m) \leq m$ for some constant $c_0 > 0$, and $D: = 8 \e^{1/2} \sqrt{\log(\e C)}$ (C from proof below)} such that \textcolor{red}{if $N \geq \widetilde{C} m$}\footnote{Kann hoffentlich wegargumentiert werden}:
\begin{enumerate}
\item[(a)] If $X$ fulfills the weak moment assumption of order $\max\{4,\log(m)\}$ with constants $\kappa_1$ and $\gamma \geq 1/2$, then there exist an absolute constant $c_0$ and constants $\widetilde{C}$ and $D$ such that if $m$ is large enough such that $4 \e^{4 \gamma} c_0^2 \kappa_1^2 \log(m) \leq m$, then with probability at least $1- 2 \exp(-2m)$, the matrix $\frac{1}{\sqrt{m}}A$ fulfills the $\ell_1$-quotient property with constant $D$ relative to the \emph{clipped $\ell_2$-norm} $\|\cdot\|^{(\sqrt{\log(\e N/m)})}:= \max\{\|\cdot \|_2,\sqrt{\log(\e N/m)} \|\cdot\|_\infty\}$ for 
$N \geq \max\big\{\widetilde{C}m,\log^{2\gamma -1}(m)\big\}$.
\item[(b)] If $X$ is super-Gaussian with parameter $0 < \sigma \leq 1$, there exist constants $\widetilde{C}$ and $D$ (depending on $\sigma$) such that with probability at least $1- 2 \exp(-2m)$, $\frac{1}{\sqrt{m}}A$ fulfills the $\ell_1$-quotient property with constant $D$ relative to the $\ell_2$-norm for $N \geq \widetilde{C} m$.
\end{enumerate}
\end{theorem}

\begin{remark}
\begin{enumerate}
\item
In the first statement of Theorem \ref{theorem_QP}, the constants $D$ and $\widetilde{C}$ depend on $\kappa_1$ and $\gamma$. In particular, $D$ can be chosen as 
\[
D = 8 \e^{-1} \sqrt{1+(9+8 \gamma) \log(4) + 8 \log(\kappa_1)} \lesssim \sqrt{\gamma} + \sqrt{\log(\kappa_1)}
\] 
and $\widetilde{C}$ as $\widetilde{C} = \frac{4^{5+8 \gamma}}{3^4 \e} \kappa_1^8$. %(old and not optimal value of the constant: 4^{9+8 \gamma} \kappa_1^8$.
\item The proof given of the second statement works for the constants $\widetilde{C} = 64^2 c_1^{-2} \e^{\frac{c_2}{16 \sigma^2}}$, $c_1^{-1}=10$, $c_2=5220$ and 
\[
D =8 \e^{1/2} \sqrt{1+ 2 \log(64)+ 2 \log(c_1^{-1}) + \frac{c_2}{16 \sigma^2}} \lesssim  1+\sigma^{-1}, %\sigma^{-1},
\]
which depend on the super-Gaussian parameter $\sigma$.
\end{enumerate}
We note that in both cases, it was not our objective to find the best possible constants $\widetilde{C}$ and $D$. By considering the case of $c m < N < \widetilde{C} m$ with much smaller $c$ than $\widetilde{C}$ separately and analyzing the smallest singular value of $B$, the range of validity of the theorem can be extended considerably, cf. also \cite[Theorem 11.19]{FoucartRauhut13}. For lower bounding the least singular values under the present random models, results as in \cite{KoltchinskiiMendelson15} are useful tools.
\end{remark}

As mentioned before, the $\ell_1$-quotient property is closely linked to the geometry of $A B_1^N$, which is the polytope defined by the absolute convex hull of the columns of $A$. We obtain the following corollary by rewriting the definition of the $\ell_1$-quotient property, see, e.g., \cite[Chapter 11]{FoucartRauhut13}.%\cite[Proposition II.15]{BA18}.
\begin{corollary} \label{cor_QP}
Let $A = (a_{ji})$ %_{1 \leq j \leq m}^{1 \leq i \leq N} \in \R^{m \times N}$ 
be an $m \times N$ random matrix with independent symmetric, unit variance entries $a_{ji} \sim X$ for all $j \in [m]$, $i \in [N]$. 
\begin{enumerate}
\item[(a)] If $X$ fulfills the weak moment assumption of order $\max\{4,\log(m)\}$ with constants $\kappa_1$ and $\gamma \geq 1/2$, then there exist an absolute constant $c_0$ and constants $\widetilde{C}$ and $D$ such that if $m$ is large enough such that $4 \e^{4 \gamma} c_0^2 \kappa_1^2 \log(m) \leq m$ and if $N \geq \max\big\{\widetilde{C}m,\log^{2\gamma -1}(m)\big\}$,
\[
\Pro \Big( A B_1^N \supset \frac{1}{D} \big(\sqrt{\log{(eN/m)}} B_2^N \cap B_\infty^N \big) \Big) \geq 1 - 2 \exp(-2 m).
\]
\item[(b)] If $X$ is super-Gaussian with parameter $0 < \sigma \leq 1$, there exist constants $\widetilde{C}$ and $D$ (depending on $\sigma$) such that for every $m,N$ satisfying $N \geq \widetilde{C} m$, one has
\[
\Pro \Big( A B_1^N \supset \frac{1}{D} \sqrt{\log{(eN/m)}} B_2^N \Big) \geq 1 - 2 \exp( -2 m).
\]
%with probability at least $1- 2 \exp(-2m)$, \CK{$\frac{1}{\sqrt{m}}A$} fulfills the $\ell_1$-quotient property with constant $D$ relative to the $\ell_2$-norm for $N \geq \widetilde{C} m$.
\end{enumerate}
\end{corollary}

%\begin{enumerate}[(a)]
%\item with i.i.d. Bernoulli entries (i.e., $b_{ij} = \pm 1$ with probability $1/2$ for all $j \in [m]$) for all $i \in [N]$.
%\end{enumerate}

\subsection{Robustness of noise-blind compressed sensing} \label{subsec_RobustnessGuarantees}
We will use \Cref{theorem_QP} to study the the robustness of the reconstruction map $\Delta_{1}$ given by equality-constrained 
$\ell_1$-minimization
\begin{equation} \label{def_ell1_constrained}
\Delta_{1}(y) := \argmin\limits_{z \in \R^N} \|z\|_{1} \text{ subject to } Az = y,
\end{equation}
where $A \in \R^{m \times N}$ for $m <N$ and when noise on the measurements $A x$ of a sparse or approximately sparse vector $x \in \R^N$ is present, i.e., if $y = Ax + w$ with some arbitrary $w \in \R^m$. Our goal is to quantify, for $1 \leq p \leq 2$, the $\ell_p$-error $\|\Delta_{1}(y) - x\|_{\ell_p}$ of the reconstruction map $\Delta_{1}(y)$ to $x$. We call the decoder $\Delta_1(y)$ \emph{noise-blind} since it does not use any information about the noise $w$.%We quantify the model accuracy of the vector $x \in \R^N$ to the set of $s$-sparse vectors by the $\ell_1$-error of the best $s$-term approximation of $x$ defined as $\sigma_s(x)_1 := \inf \{\|x-z\|_1, z\in \R^N \text{ is }s\text{-sparse}\}$.
%With this definition, we the desired robust recovery guarantees can be expressed as upper bounds of $\|\Delta_{1}(y) - x\|_{\ell_p}$ relative to $\sigma_s(x)_1$ and a suitable norm on the noise vector $w$.

Furthermore, a more canonical reconstruction algorithm in case of noisy observations ($w \neq 0$) is the convex program called \emph{quadratically constrained $\ell_1$-minimization}
\begin{equation} \label{def_ell1_QC}
\Delta_{1,\eta}(y) := \argmin\limits_{z \in \R^N} \|z\|_{1} \text{ subject to } \|Az- y\|_2 \leq \eta
\end{equation}
for some $\eta > 0$. The parameter $\eta$ can be chosen in a \emph{noise-aware} manner such that $\|w\|_2 \leq \eta$, using oracle information about the $\ell_2$-norm of the noise $w$, and error bounds such as \cref{rec:bound:noisy} have been shown by using a restricted isometry property or robust null space property of $A$ in this case, without the need of using quotient properties.

In the next theorem, we derive error bounds for $\Delta_1(y)$ and for $\Delta_{1,\eta}(y)$ in the case of underestimated  noise level such that $\eta < \|w\|_2$. The latter was first studied in \cite{BA18}. The theorem provides robustness results for measurement matrices drawn from a wide range of i.i.d. entrywise distributions with high probability.

%The equality-constained $\ell_1$-minimization of \cref{def_ell1_constrained} can be seen as a limit case of \cref{def_ell1_QC} for $\eta \to 0$. From a practical perspective, the decoder of \cref{def_ell1_QC} has the fundamental disadvantage that it is \emph{noise-aware}, i.e., it uses oracle-information about the level of the noise through the parameters $\eta$, unlike \cref{def_ell1_constrained}.
%Related to the robustness of $\Delta_1$ is the question of robustness of $\Delta_{1,\eta}$ if the noise level is not accurately estimate, i.e., if $\|w\|_2 > \eta$, which has already been studied in \cite{BA18}.
%}
%To describe quite general sufficient conditions that allow for robustness of $\Delta_{1}$, we use the following definitions. %\textcolor{red}{Add: Known results by Wojtaszczyk and Foucart, or put into related work section.}
%
%(formerly: Definition of r.v. types here.)
%
%We establish the following theorem that provides a robustness bound for $\Delta_{1} = \Delta_{1,0}$ and, if $\eta > 0$, for $\Delta_{1,\eta}$ in the case of underestimated noise such that $\|w\|_2 > \eta$, with high probability, provided the measurement matrix is drawn from a wide range of i.i.d. entrywise distributions. We note that the assumption of identical distributions can be relaxed, but for simplicity we present the theorem under this assumption.
\begin{theorem} \label{theorem_robustnessGuarantees}
Let $m \leq N$, let $B = (b_{ji})$ %_{1 \leq j \leq m}^{1 \leq i \leq N} \in \R^{m \times N}$
be an $m \times N$ random matrix with independent symmetric, unit variance entries $b_{ji} \sim X$ for all $j \in [m]$, $i \in [N]$ and $A := \frac{1}{\sqrt{m}}B$. For $s \in \N$, let $\sigma_s(x)_1 := \inf \{\|x-z\|_1: z\in \R^N, \|z\|_0 \leq s\}
$ be the $\ell_1$-error of the best $s$-term approximation of $x \in \R^m$. Assume $\eta \geq 0$.
\begin{enumerate}
\item[(a)] Assume that $X$ fulfills the weak moment assumption of order $\max\{4,\log(N)\}$ with constants $\kappa_1$ and $\gamma \geq \frac{1}{2}$. Then there exist constants $\widetilde{c}_1, \widetilde{c}_2, \widetilde{c}_3, \widetilde{C}, C, D, E >0$ depending only on $\kappa_1$ and $\gamma$ such that if
\[
N \geq \widetilde{C} m, \quad m \geq \widetilde{c}_3 \log^{\max\{2\gamma -1,1\}} (N) \quad \text{ and } s \leq \widetilde{c}_2 s_* := \widetilde{c}_2 \frac{m}{\log(\e N/m)},
\] 
with probability at least $1 - 3 \exp(-\widetilde{c_1} m)$, the solution of the $\ell_1$-minimization decoder $\Delta_{1,\eta}$ given the measurement matrix $A$ and data vector $y= Ax+w$  fulfills the $\ell_p$-error estimates
\[
\begin{split}
&\| x - \Delta_{1,\eta}(Ax + w)\|_{p} \\
\leq &\frac{C}{s^{1-1/p}} \sigma_s (x)_1+ s_*^{1/p-1/2} \Big(D \eta + E \frac{\|w\|^{(\sqrt{\log{\e N/m}})}}{\|w\|_2} \max\{\|w\|_2 - \eta,0\}\Big)
\end{split}
\] % old version (for equality-constrained $\ell_1$-min only): \| x - \Delta_{1}(Ax + w)\|_{p} \leq \frac{C}{s^{1-1/p}} \sigma_s (x)_1+ D s_*^{1/p-1/2} \|w\|^{(\sqrt{\log{\e N/m}})}
for $1 \leq p \leq 2$, for all $x \in \R^{N}$ and all $w \in \R^m$, where we recall that \\ $\|\cdot\|^{(\sqrt{\log(\e N/m)})}:= \max\{\|\cdot \|_2,\sqrt{\log(\e N/m)} \|\cdot\|_\infty\}$.
\item[(b)] If $X$ fulfills the weak moment assumption of order $\max\{4,\log(N)\}$ with constants $\kappa_1$ and $\gamma \geq \frac{1}{2}$ and if $X$ is additionally super-Gaussian with parameter $0 \leq \sigma \leq 1$, then there exist constants $\widetilde{c}_1, \widetilde{c}_2, \widetilde{C}, C, D, E >0$ depending only on $\kappa_1$, $\gamma$ and $\sigma$ such that if
\[
N \geq \widetilde{C} m, \quad m \geq \log^{\max\{2\gamma -1,1\}} (N) \quad \text{ and } s \leq \widetilde{c}_2 s_* := \widetilde{c}_2 \frac{m}{\log(\e N/m)},
\]
with probability at least $1 - 3 \exp(-\widetilde{c_1} m)$, the solution of the equality-constrained $\ell_1$-minimization problem $\Delta_{1}(Ax + w)$ fulfills the $\ell_p$-error estimates
\[
\| x - \Delta_{1,\eta}(Ax + w)\|_{p} \leq \frac{C}{s^{1-1/p}} \sigma_s (x)_1+ s_*^{1/p-1/2} \big(D \eta + E \max\{\|w\|_2 -\eta,0\}\big)
\]
for $1 \leq p \leq 2$, and for all $x \in \R^{N}$, $w \in \R^m$.
\end{enumerate}
\end{theorem}
\begin{remark} \label{remark_7}
We point out that robust recovery guarantees of \Cref{theorem_robustnessGuarantees} can be specified for the noise-blind \emph{equality-constrained} $\ell_1$-minimization 
\[
\Delta_{1,0}(y) = \Delta_{1}(y) = \argmin\limits_{z \in \R^N} \|z\|_{1} \text{ subject to } Az = y
\]
such that for all $1 \leq p \leq 2$,
\[
\| x - \Delta_{1}(Ax + w)\|_{p} \leq \frac{C}{s^{1-1/p}} \sigma_s (x)_1+ s_*^{1/p-1/2} E \|w\|^{(\sqrt{\log{\e N/m}})}
\]
and 
\[
\| x - \Delta_{1,\eta}(Ax + w)\|_{p} \leq \frac{C}{s^{1-1/p}} \sigma_s (x)_1+ s_*^{1/p-1/2}E \|w\|_2
\]
for all $x\in \R^N$ and $w \in \R^m$ in the first and second part of the theorem, respectively.

This means that the existing robust recovery guarantees for this decoder using matrices with i.i.d.\ sub-Gaussian \cite[Theorem 11.10]{FoucartRauhut13}, Gaussian \cite[Theorem 11.9]{FoucartRauhut13} and Weibull \cite[Theorem 11]{Foucart14} random variables can be considered as special cases of the theorem.
%
%\HR{In the equality-constrained case $\eta = 0$, and $p=2$, say, the error bounds read as follows: Under the
%assumptions in 1.\ we have
%\[
%\|x-\Delta_1(Ax + w)\|_2 \leq C \frac{\sigma_s(x)_1}{\sqrt{s}} + E \|w\|^{\big(\sqrt{\log(eN/m)}\big)},
%\]
%while in the situation of 2.\ we obtain
%\[
%\|x-\Delta_1(Ax + w)\|_2 \leq C \frac{\sigma_s(x)_1}{\sqrt{s}} + E \|w\|_2.
%\]
%The important point is that the right hand sides feature the true norm of $w$ rather than only an upper estimate
%as in the standard estimates \eqref{rec:bound:noisy}.
%}
%
\end{remark}

To show Theorem \ref{theorem_robustnessGuarantees}, we combine existing results about the \emph{robust null space property} of matrices with i.i.d. entries drawn from distributions fulfilling weak moment assumptions \cite{DLR16, ML17} together with \Cref{theorem_QP}.

Next, we illustrate the generality of the assumptions of Theorems \ref{theorem_robustnessGuarantees} and \ref{theorem_QP} by enumerating random models which are covered by our theorem, but which mostly have not been covered by the robustness analyses of \cite{Litvak05,Wojtaszczyk10,Foucart14}.
\begin{example} \label{examples_1}
Let $A$ be a real random matrix with i.i.d. entries $A_{ji}= \frac{1}{\sqrt{m}}  \frac{B_{ji}}{\|B_{ji}\|_{L_2}}$, $j \in [m]$, $i \in [N]$.
\begin{enumerate}[label=(\roman*)]
\item Assume that the $B_{ji}$ are distributed as $X_\gamma$, where $X_\gamma$ is a $\psi_{\frac{1}{\gamma}}$-random variable with the same distribution as $\sign(g)|g|^{2 \gamma}$, where $g$ is a standard normal variable and $\gamma > 0$. Then, $X_\gamma$ is of exponential type, i.e., has a probability density function of $p(x) = c_1 \e^{-\frac{|x|^\frac{1}{\gamma}}{c_2}}$. If $1/2 \leq \gamma \leq 1$, the assumptions of the second part of Theorem \ref{theorem_robustnessGuarantees} apply, cf. \cite[Example V.4]{DLR16}. In particular, the $B_{ji}/\|B_{ji}\|_{L_2}$ are super-Gaussian with a parameter $\sigma \geq \frac{1}{2}$. We note that the special cases $\gamma = \frac{1}{2}$ and $\gamma = 1$ have been covered already by the existing theory, in the latter case by \cite{Foucart14}, but not for $1/2 < \gamma < 1$.

For $\gamma > 1$, the theorem still applies as $X_\gamma$ is super-Gaussian with parameter $\sigma \geq \frac{1}{2}$, but with a worse upper bound $\widetilde{c}_2(\gamma) \frac{m}{\log(\e N/m)}$ on the sparsity, where $\widetilde{c}_2(\gamma)$ depends on $\gamma$, cf. \cite[Example V.4]{DLR16}. \label{example_distributions_1}

\item Let $d \in \N$. If the $B_{ji}$ are distributed as Student-$t$ variables $X_d$ with $d$ degrees of freedom, they are (after normalization) super-Gaussian with some parameter $\sigma \geq \frac{1}{2}$ as well, so that Theorem \ref{theorem_QP}(b) applies. If additionally $d \geq 2 \log(N)$, then also Theorem \ref{theorem_robustnessGuarantees}.2 applies.
\item If the $B_{ji}$ are distributed as a symmetric Weibull variable $X_r$ with exponent $1 \leq r \leq 2$, recovery guarantees or equality-constrained $\ell_1$-minimization that are robust relative to the $\ell_2$-norm have been shown in the optimal regime of $m$ already in \cite{Foucart14}. Since the normalized symmetric Weibull variables $B_{ji}/\|B_{ji}\|_{L_2}$ are super-Gaussian with parameter $\sigma \geq \frac{1}{2}$ for all $1 \leq r \leq 2$, Theorem \ref{theorem_QP}(b) applies also here.
\end{enumerate}
\end{example}

Interestingly, comparing the two parts of \Cref{theorem_robustnessGuarantees}, we see that our analysis suggests that the robustness properties of equality-constrained $\ell_1$-minimization $\Delta_{1}$ with measurement matrices $A$ with entries drawn from many super-Gaussian distributions are \emph{asymptotically better} than the ones of measurement matrices whose entries are drawn from certain sub-Gaussian, bounded distributions as the Rademacher distribution with random signs.

\section{Numerical experiments}
\label{sec:numerical}
In this section, we show in a case study that the results of \Cref{theorem_robustnessGuarantees} give an appropriate explanation of the empirical robustness behavior of different measurement matrices. In particular, we consider three types of measurement matrices: random matrices with i.i.d. Gaussian, Bernoulli and Student-t entries. The presented numerical experiments have been conducted using MATLAB R2017b on a MacBook Pro with a 2.3 GHz Intel Core i5 processor. The convex optimization problems of our experiments are solved using the CVX package \cite{cvx}.

\subsection{Behavior under spherical noise} \label{subsec_SphericalNoise}
\begin{figure}[t]
\includegraphics[width=1\textwidth]{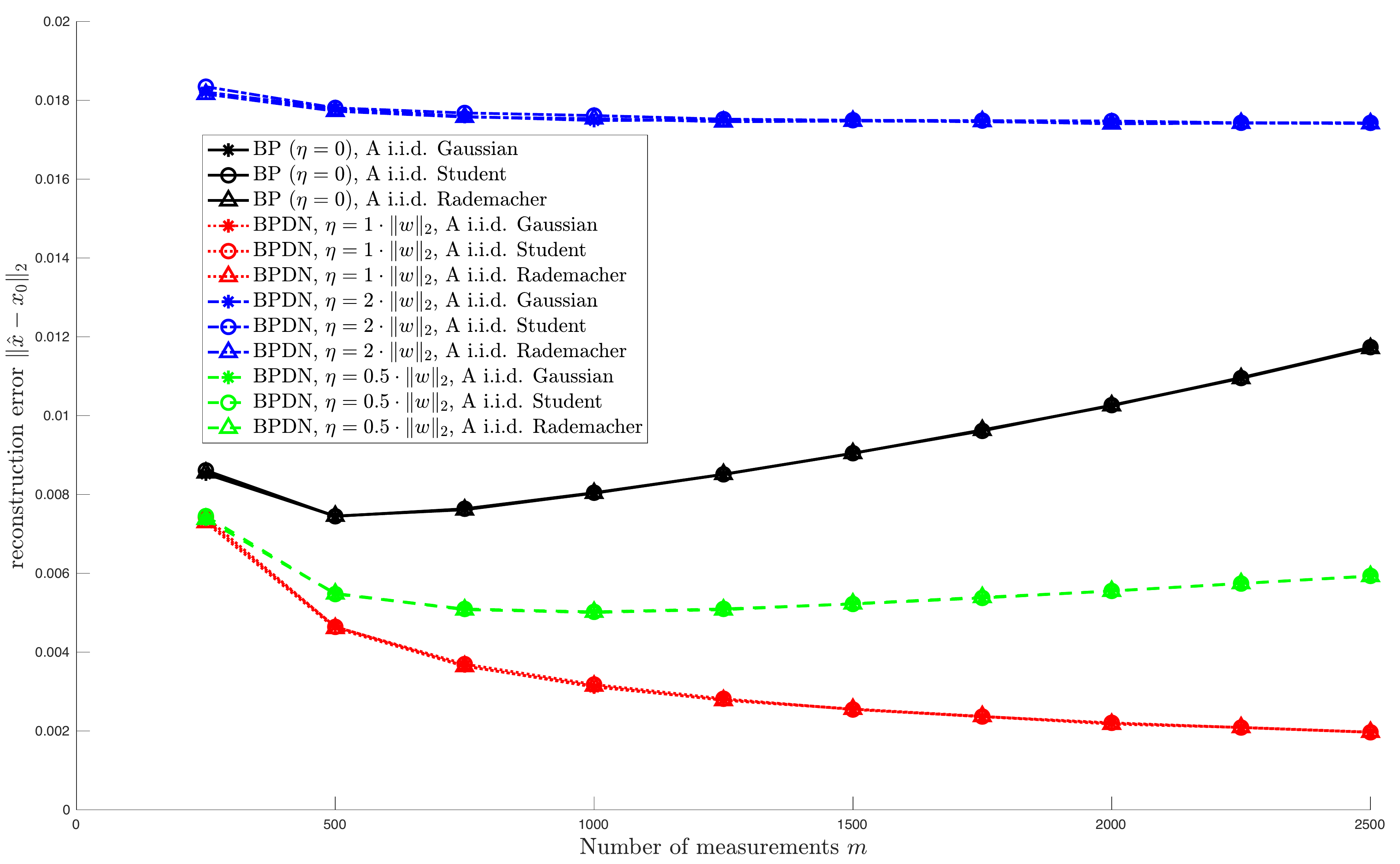}%{QP_numerics_Gaussian_noise_1_cut.pdf}
\caption{y-axis: Reconstruction errors $\|\widehat{x}-x_0\|_{2}$ from measurements $y=Ax_0 + w$, $w$ is a random spherical noise vector with $\|w\|_{2}= 10^{-2}$, for different measurement matrices $A$. \\
x-axis: number of rows of $m$.} \label{figure_GaussNoise}
\end{figure}
%\subsection{Behavior under spherical noise} \label{subsec_SphericalNoise}
%\begin{figure}[t]
%\includegraphics[width=1\textwidth]{QP_numerics_Gaussian_noise_1_cut.pdf}
%\caption{y-axis: Reconstruction errors $\|\widehat{x}-x_0\|_{2}$ from measurements $y=Ax_0 + w$, $w$ is a random spherical noise vector with $\|w\|_{2}= 10^{-2}$, for different measurement matrices $A$. x-axis: number of rows of $m$.} \label{figure_GaussNoise}
%\end{figure}
In our first experiment, we perform simulations for the reconstruction of a $s$-sparse vector $x_0 \in \R^N$ with $\|x_0\|_2 = 1$ from measurements $y= Ax_0 + w$ which are perturbed by a random vector $w \in \R^m$ that is drawn from the uniform distribution on the sphere of radius $\|w\|_2 = 10^{-2}$. To obtain our reconstruction result $\widehat{x}$, we use equality-constrained $\ell_1$-minimization \cref{def_ell1_constrained} as defined by $\Delta_{1}(y)$ and quadratically constrained $\ell_1$-minimization \cref{def_ell1_QC}
\begin{equation*}
\Delta_{1,\eta}(y) := \argmin\limits_{z \in \R^N} \|z\|_{1} \text{ subject to } \|Az - y\|_2 \leq \eta,
\end{equation*}
where the noise level estimate $\eta$ is chosen such that $\eta \in \{\|w\|_2, 2 \|w\|_2, 0.5 \|w\|_2\}$, i.e., the noise level $\|w\|_2$ is either estimated accurately or over- or underestimated by a factor of two.
The support $S$ of $x_0$ is drawn uniformly among the $\binom{N}{s}$ possibilities, and the non-zero coordinates are drawn uniformly on the sphere $\mathcal{S}^{S-1} = \{ x \in \R^N: \|x\|_2=1, \supp(x) \subset S\}$.

In \Cref{figure_GaussNoise}, the resulting recovery $\ell_2$-errors $\|\widehat{x}-x_0\|_2$ can be observed for the three different random models (in case of Student-t measurements, $k=9$ degrees of freedoms were used) for the measurement matrix $A$ mentioned above, where the parameters were chosen as $N=5000$, $s= 10$ and $m \in \{\lceil k N/20 \rceil, k=1,\ldots,14\}$. The reported errors are averaged over $500$ runs of the simulation.

We notice that in the experiment, the recovery error of the equality-constrained algorithm \cref{def_ell1_constrained} is comparable to the one of quadratically constrained $\ell_1$-minimization \cref{def_ell1_QC} with correctly estimated or underestimated noise level $\eta \in \{\|w\|_2, 0.5 \|w\|_2\}$, if $A$ has a small number of rows $m \leq 500$. For larger $m$, the robustness of \cref{def_ell1_QC} improves further if $\eta = \|w\|_2$, whereas it stagnates for underestimated noise of $\eta = 0.5 \|w\|_2$ and it deteriorates slightly for \cref{def_ell1_constrained}.

It can be also observed that an overestimation of the noise level such that $\eta = 2 \|w\|_2$ in \cref{def_ell1_QC} leads to a significantly worse reconstruction error $\|\widehat{x}-x_0\|_{2}$ than for all the other methods, for all the considered number of measurements $m$. 

Importantly, we observe that the robustness behavior of the algorithms does not depend on the choice of Gaussian, Bernoulli or Student-t measurement matrices in this case of presence of spherical noise.

This is precisely in accordance to the result of \Cref{theorem_robustnessGuarantees}: Bernoulli variables $X$ fulfill the assumptions of the first part of the theorem, but not of the second part, since they are sub-Gaussian. On the other hand, Gaussian and Student-t variables (with a sufficient number of degrees of freedom) fulfill the assumptions for the stronger statement of \Cref{theorem_robustnessGuarantees}.2. In general, Bernoulli measurement matrices entail the weaker statement predicting a reconstruction error of 
\[
\| x - \Delta_{1}(y)\|_{2} \leq D \|w\|^{(\sqrt{\log{\e N/m}})}
\]
with a constant $D$ for equality-constrained $\ell_1$-minimization. For spherical noise, though, this coincides with the statement of \Cref{theorem_robustnessGuarantees}.2, since $\|w\|^{(\sqrt{\log{\e N/m}})}= \|w\|_2$ with high probability under this noise model.

\subsection{Behavior under heavy-tailed noise}
\begin{figure}[t]
\includegraphics[width=1\textwidth]{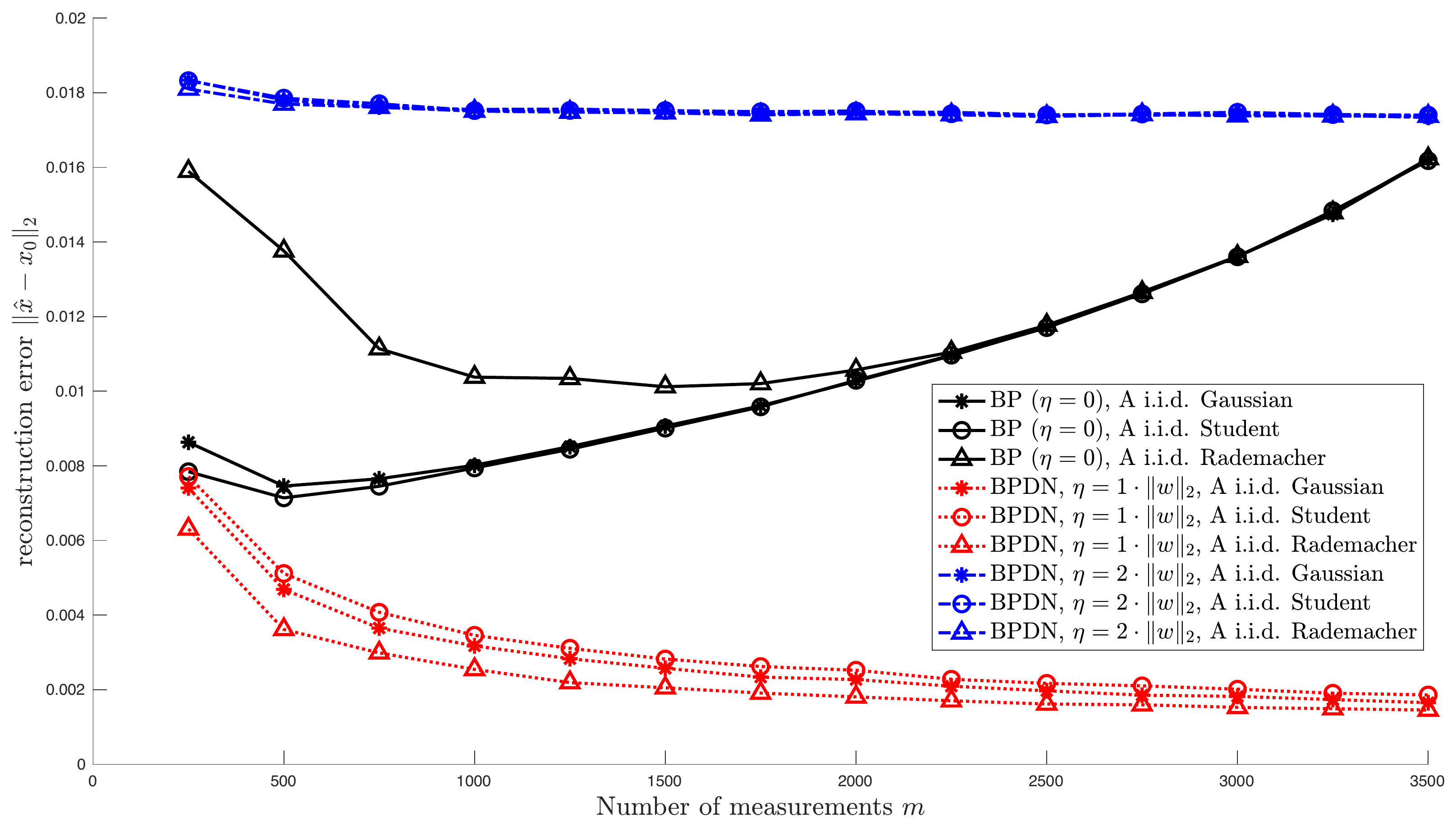}
\caption{y-axis: Reconstruction errors $\|\widehat{x}-x_0\|_{2}$ from measurements $y=Ax_0 + w$, $w= 10^{-2} \cdot \frac{\widetilde{w}}{\|\widetilde{w}\|_2}$ where the entries of $\widetilde{w}$ are i.i.d. $\psi_{0.2}$ random variables, for different measurement matrices $A$. \\
x-axis: number of rows of $m$.} \label{figure_SparseNoise}
\end{figure}

Next, instead of uniform spherical noise, we consider more heavy-tailed noise such that $w = 10^{-2} \frac{\widetilde{w}}{\|\widetilde{w}\|_2} \in \R^m$, where $(\widetilde{w})_i$ are i.i.d. $\psi_{\alpha}$ random variables for the parameter $\alpha = 0.2$, cf. also \Cref{examples_1}.\ref{example_distributions_1}. Such a noise has most of its mass in just few coordinates, and the size of its largest entry $\|w\|_\infty$ is comparable to its $\ell_2$-norm $\|w\|_2$, i.e.  $\|w\|_{\infty} \approx \|w\|_{2}=10^{-2}$ with high probability. 

In this case, the conclusions about the recovery accuracy of equality-constrained \sloppy{$\ell_1$-minimization} \cref{def_ell1_constrained} that can be drawn from \Cref{theorem_robustnessGuarantees}.1 and  \Cref{theorem_robustnessGuarantees}.2 predict a better behavior of Gaussian and Student-t measurements than for Bernoulli measurements, in particular for $m \ll N$: Since then
\[
\|w\|^{(\sqrt{\log{\e N/m}})}= \sqrt{\log{\e N/m}}\cdot  \|w\|_{\infty} \approx \sqrt{\log{\e N/m}}\cdot \|w\|_{2} = \sqrt{\log{\e N/m}}  \cdot 10^{-2}
\]
with high probability, \Cref{theorem_robustnessGuarantees} predicts a reconstruction error of 
\[
\| x - \Delta_{1}(y)\|_{2} \leq D \cdot \sqrt{\log{\frac{\e N}{m}}} \cdot 10^{-2}
\]
for Bernoulli measurements, but a reconstruction error of 
\[
\| x - \Delta_{1}(y)\|_{2} \leq D \cdot  10^{-2}
\]
for the two other, more heavy-tailed measurement models (here, $D$ is some constant).

These predictions can be well confirmed in the experiment illustrated in \Cref{figure_SparseNoise}, repeating the experiment from \Cref{subsec_SphericalNoise} for this different, heavy-tailed noise model: 
Unlike before, the reconstruction error of equality-constrained $\ell_1$-minimization \cref{def_ell1_constrained} for Bernoulli matrices is now consistently worse than for the Gaussian and Student-$t$ measurement matrices if $m \ll N$, i.e., if $m = 250,\ldots,2000$. It is interesting to note that equality-constrained $\ell_1$-minimization with Student-t matrices (with $k = \lceil \log(N)\rceil = 9$ degrees of freedom) is even slightly more robust than  in the case that Gaussian matrices are used, especially if $m$ is small. 

On the other hand, the relative performance of Student-t measurements is worse than the one of Gaussian measurements if the noise-aware quadratically constrained $\ell_1$-minimization \cref{def_ell1_QC} is used as a reconstruction algorithm.

As for spherical noise, we also note here that overestimating the noise level by a factor of two ($\eta = 2\cdot \|w\|_2$) in \ref{def_ell1_QC} leads to worse reconstructions than the noise-blind usage of \cref{def_ell1_constrained}.

We want to stress two conclusions from these experiments:
\begin{itemize}
\item The noise-blind reconstruction algorithm \cref{def_ell1_constrained} is at least as robust in presence of certain heavy-tailed measurement matrices as in the case of Gaussian measurement matrices, especially if the measurement matrix has few rows $m$.
\item While a very precise choice in the noise level estimate $\eta$ of \cref{def_ell1_QC} leads to better reconstructions than using the noise-blind variant \cref{def_ell1_constrained}, the reconstructions deteriorate quickly once $\eta$ is chosen as an overestimate of the actual noise level. In this sense, it is preferred to choose an underestimated $\eta$ or even $\eta=0$ (resulting again in \cref{def_ell1_constrained}) in situations where there is little a priori knowledge about the noise $w$.
\end{itemize}

\section{Proof of the $\ell_1$-quotient property and of the robustness of noise-blind $\ell_1$-minimization} \label{sec_proofs}
In this section, we provide proofs of \Cref{theorem_QP} and \Cref{theorem_robustnessGuarantees}. As a first step, we provide characterizations of the clipped $\ell_2$-norm $\|y\|^{(\alpha)} = \max\{ \|y\|_2, \alpha \|y\|_{\infty}\}$ of a vector $y \in \R^m$ for $\alpha \geq 1$ and also of its dual norm $\|\cdot\|_*^{(\alpha)}$ in some preliminary lemmas. A sufficient condition for a matrix to $A$ to fulfill the $\ell_1$-quotient property relative to a general norm is provided in Lemma \ref{lemma_QP_equivalence}. Then, we present probabilistic arguments for this condition relative to clipped norms using results derived from \emph{Mendelson's small ball method} \cite{MendelsonLearning15,KoltchinskiiMendelson15,DLR16} by bounding appropriate quantities related to the distribution in question, which constitutes the main part of the proof.

\subsection{Preliminary lemmas}
Recall that for $\alpha \geq 1$, we defined the \emph{clipped $\ell_2$-norm with parameter $\alpha$} of $y \in \R^m$ as
\[
\|y\|^{(\alpha)} := \max\{ \|y\|_2, \alpha \|y\|_{\infty}\}.
\]
We will use the following two lemmas about its dual norm $\|y\|_*^{(\alpha)}= \sup_{\|z\|^{(\alpha)}\leq 1} |\langle x,z\rangle|$ that can be found in \cite[Lemma 11.22]{FoucartRauhut13} and \cite[Lemma 2]{MS90}. They provide an explicit formula for $\|\cdot\|_*^{(\alpha)}$ and compare it with the norm $\| \cdot \|_{\alpha^2,\dagger}$ defined below, whose advantage will become clear later on.
\begin{lemma}
Let $T :=T(y):= \{ i \in [m] : \alpha |y_i| \leq \|y\|_2\}$ and \\
$T^c = \left\{ i \in [m] : |y_i| > \frac{1}{\alpha} \|y\|_2\right\}$.

Then
\begin{equation} \label{eq_def_dualclipped}
\|y\|_*^{(\alpha)} = \|y_T\|_2 + \frac{1}{\alpha} \|y_{T^c}\|_1
\end{equation}
for all $ y \in \R^m$.
\end{lemma}

\begin{lemma} \label{lemma_norm_alpha_star} %[{\cite[Lemma 11.22]{FoucartRauhut13} and originally \cite[Lemma 2]{MS90}}]
Assume $\alpha^2$ is an integer. Then the dual norm $\|\cdot \|_*^{(\alpha)}$ of $\|\cdot \|^{(\alpha)}$ is comparable with the norm $\| \cdot \|_{\alpha^2,\dagger}$ defined by
\begin{equation} \label{def_dagger_norm}
\| y \|_{\alpha^2,\dagger} := \max \left\{ \sum_{\ell =1 }^{\alpha^2} \|y_{B_\ell}\|_2, B_1,\ldots,B_{\alpha^2} \text{ form a partition of }[m] \right\}
\end{equation}
in the sense that
\[
\frac{1}{\alpha} \| y \|_{\alpha^2,\dagger} \leq \| y \|_*^{(\alpha)} \leq \frac{\sqrt{2}}{\alpha}  \| y \|_{\alpha^2,\dagger}
\]
for all $y \in \R^m$.
\end{lemma}

We note that for $\alpha =1$, the clipped norm and its dual norm reduce to the $\ell_2$-norm since $\|y\|_{\infty} \leq \|y\|_2$, i.e., $\|y\|^{(\alpha)} = \|y\|_{2} = \|y\|_*^{(\alpha)}$ for all $y \in \R^m$ if $\alpha =1$.

The following two lemmas provide a reformulation of the $\ell_1$-quotient property relative to a norm $\| \cdot \|$, which will be more convenient to analyze.
\begin{lemma}[{\cite[Lemma 11.17]{FoucartRauhut13}}] \label{lemma_FR_11_13}
%Let $q \geq 1$. 
A matrix $A \in \R^{m \times N}$ has the $\ell_1$-quotient property with constant $d$ relative to $\| \cdot \|$ if and only if
\[
\| w \|_* \leq d s_*^{1/2} \|A^* w \|_{\infty} \quad \text{ for all } w \in \R^m,
\]
where $s_* = m / \log(\e N/m)$.
% and where $q^* = \frac{q}{q-1}$ is the conjugate exponent of $q$ (with the understanding that $q^* = \infty$ if $q=1$).
\end{lemma}

\begin{lemma} \label{lemma_QP_equivalence}
Let $A = (a_{ji}) \in \R^{m \times N}$ %_{\substack{1 \leq i \leq m \\ 1\leq j \leq N}} \in \R^{m \times N}$ 
be a matrix with columns $a_i \in \R^m$, $i \in [N]$ and $q=\log(N)$. If 
\begin{equation} \label{eq_qp4}
\inf_{w \in S^{\|\cdot\|_{*}} } \left(\frac{1}{N} \sum_{i=1}^N |\langle a_i,w \rangle|^q\right)^{1/q} \geq  \frac{1}{D} \sqrt{\log\Big(\frac{\e N}{m}\Big)},
\end{equation}
where $S^{\|\cdot\|_{*}} = \{ w \in \R^m : \|w\|_* = 1 \}$ is the unit sphere of the dual norm $\|\cdot\|_*$ of some norm $\|\cdot\|$, then $\frac{1}{\sqrt{m}}A$ fulfills the $\ell_1$-quotient property with constant $D$ relative to the norm $\|\cdot\|$.
\end{lemma}
\begin{proof}
Let $w \in \R^m$. Then with $q=\log(N)$, 
\[
\begin{split}
\big\|\Big(\frac{1}{\sqrt{m}} A\Big)^* w\big\|_{\infty} \geq \frac{1}{\mathrm{e}\sqrt{m}}\|A^* e\|_{\log N} &= \frac{1}{\e\sqrt{m}} \left(\sum_{i=1}^N |\langle a_i,w \rangle|^q\right)^{1/q} \\
 &= \frac{1}{\sqrt{m}} \left(\frac{1}{N}\sum_{i=1}^N |\langle a_i,w \rangle|^q\right)^{1/q},
\end{split}
\]
as $N^{1/q}= N^{1/\log(N)} = \e$. It follows from \Cref{lemma_FR_11_13} that the $\ell_1$-quotient property of $\frac{1}{\sqrt{m}}A$ with constant $D$ relative to the norm $\|\cdot\|$ is implied by
\[
\|A^* w\|_{\infty} \geq \frac{1}{D} \sqrt{\log(\e N/m)} \|w\|_* \quad \text{ for all }w\in \R^m,
\]
where $\|\cdot\|_*$ is the dual norm of the norm $\|\cdot\|$. This implies that
\[
\left(\frac{1}{N}\sum_{i=1}^N |\langle a_i,w \rangle|^q\right)^{1/q} \geq \frac{1}{D} \sqrt{\log(\e N/m)}\|w\|_*  \quad\text{ for all }w\in \R^m
\] 
is a sufficient condition for the assertion of the lemma.
\end{proof}

\subsection{Application of the small-ball method}
The following result due to \cite[Lemma III.1]{DLR16} and \cite[Theorem 1.5]{KoltchinskiiMendelson15} will be used to show the sufficient condition of Lemma \ref{lemma_QP_equivalence}. 
\begin{lemma} \label{theoremlemma31}
Let $1 \leq q < \infty$, let $S \subset \R^m$ be a set and $b_1, \ldots, b_N$ be i.i.d. copies of a random vector $b$ in $\mathbb{R}^m$. For $u > 0$, define
\[
Q_S (u) := \inf_{w \in S} \mathbb{P}(| \langle b,w \rangle | \geq u)
\]
and
\begin{equation} \label{eq_complexity_para}
\mathcal{R}_N (S) := \mathbb{E}\left[ \sup_{w \in S} \left| \frac{1}{N} \sum_{i=1}^N \epsilon_i \langle b_i, w\rangle \right|\right],
\end{equation}
where $(\epsilon_i)_{i \in [N]}$ is a Rademacher sequence that is independent from $(b_i)$.
%where $(\epsilon_i)_{1 \leq i \leq N}$ is a Rademacher sequence that is independent from the $b_i$, i.e., $\Pro(\epsilon_i) = \pm 1) = \frac{1}{2}$ with indepdent $\epsilon_i$ for $i \in [N]$. If 
Then, for $t > 0$, with probability at least $1- 2 \mathrm{e}^{-2 t^2}$,
\begin{equation} \label{eq_small_ball_lemma}
\inf_{w \in S} \frac{1}{N} \sum_{i=1}^N |\langle b_i, w\rangle|^q \geq u^q \left(Q_S (2u) - \frac{4}{u} \mathcal{R}_N (S) - \frac{t}{\sqrt{N}} \right).
\end{equation}
\end{lemma}
To prove our main results, Lemma \ref{theoremlemma31} is used by choosing the set $S = S^{\|\cdot \|_*^{(\beta)}}$, where $S^{\|\cdot \|_*^{(\beta)}}$ is the unit sphere of the dual of the clipped norm with $\beta = \sqrt{c_1 \log\left(\frac{eN}{m}\right)}$ and $\beta = 1$, respectively, and $c_1 > 0$ being an absolute constant. We show the following lemma.

%\subsection{Upper bound for $\mathcal{R}_N (B^{\|\cdot \|_*^{(\alpha)}})$} \label{sec_upperbd}
%Let $\alpha \geq 1$. In order to apply \Cref{theoremlemma31} for showing \cref{eq_qp4}, we now find an upper bound for the complexity parameter
%\begin{equation} \label{def_R_N_B}
%\mathcal{R}_N (B^{\|\cdot \|_*^{(\alpha)}}) = \mathbb{E}\left[ \sup_{y \in B^{\|\cdot \|_*^{(\alpha)}}} \left| \frac{1}{N} \sum_{i=1}^N \epsilon_i \langle X_i, y\rangle \right|\right],
%\end{equation}
%where $(\epsilon_i)_{i = 1}^{N}$ is a sequence of i.i.d. Rademachers and $X_1,X_2,\ldots, X_N$ i.i.d. random vectors in $\R^m$ with i.i.d. entries distributed as $X$, whose distributions are also independent from the $\epsilon_i$, $i \in [N]$. We show the following lemma.

\begin{lemma} \label{lemma_upperbd_R}
Let $b_1,\ldots,b_N$ be distributed as i.i.d.\ vectors in $\R^m$ with independent entries $b_{ij} \sim X$, where $X$ is a variable such that $\|X\|_{L_2}=1$. Let $S^{\|\cdot \|_*^{(\alpha)}} = \{ w \in \R^m \mid \|w\|_*^{(\alpha)} =1\}$ be the unit sphere of the norm $\|\cdot \|_*^{(\alpha)}$ for $\alpha \geq 1$.
\begin{enumerate}
\item If $\alpha = 1$, then the complexity parameter $\mathcal{R}_N (S^{\|\cdot \|_*^{(\alpha)}})$ from \cref{eq_complexity_para} fulfills
\[
  \mathcal{R}_N (S^{\|\cdot \|_*^{(\alpha)}}) = \mathcal{R}_N (S^{\|\cdot \|_2}) \leq \sqrt{\frac{m}{N}}.
\]
\item Assume $\alpha > 1$. If $X$ fulfills the weak moment assumption of order $\max\{4,\log(m)\}$ with constants $\kappa_1$ and $\gamma$ and if $N \geq (\log(m))^{2 \gamma -1}$, then the complexity parameter $\mathcal{R}_N (S^{\|\cdot \|_*^{(\alpha)}})$ from \cref{eq_complexity_para} fulfills
\[
 \mathcal{R}_N (S^{\|\cdot \|_*^{(\alpha)}}) \leq\frac{1}{\sqrt{N}}\big(\sqrt{m} + \alpha \e^{2\gamma}  c_0\kappa_1  \sqrt{\log(m)}\big) % \frac{1}{\sqrt{N}} \mathbb{E}\left[ \max\{ \|h \|_2, \alpha \|h\|_{\infty}\} \right],
\]
for an absolute constant $c_0 >0$.
\end{enumerate}
%In the case that $\alpha =1$,
%Furthermore, there exists an absolute constant $c_0 > 0$ such that 
%\[
% \mathcal{R}_N (B^{\|\cdot \|_*^{(\alpha)}}) \leq 
% \begin{cases} 
% \sqrt{\frac{m}{N}}, & \text{ if }\alpha = 1, \\
% \frac{1}{\sqrt{N}}\big(\sqrt{m} + \alpha \e^{2\gamma}  c_0\kappa_1  \sqrt{\log(m)}\big) & \text{ if }\alpha >1.
% \end{cases}
%\]
\end{lemma}
\begin{proof}
Let $h := \frac{1}{\sqrt{N}} \sum_{i=1}^N \epsilon_i b_i$, 
where $(\epsilon_i)_{i=1}^m$ is a Rademacher sequence
independent of $(b_i)$. %where $(\epsilon_i)_{i=1}^m$ where the as in the definition of $\mathcal{R}_N (B^{\|\cdot \|_*^{(\alpha)}})$ in \cref{eq_complexity_para}. 
Then we obtain
\begin{equation} \label{eq_R_1}
\begin{split}
\mathcal{R}_N (B^{\|\cdot \|_*^{(\alpha)}}) &= \mathbb{E}\left[ \sup_{w \in B^{\|\cdot \|_*^{(\alpha)}}} \Big| \left\langle w, \frac{1}{N} \sum_{i=1}^N \epsilon_i b_i\right\rangle \Big| \right] = \mathbb{E}\left[ \left\| \frac{1}{N} \sum_{i=1}^N \epsilon_i b_i \right\|^{(\alpha)} \right] \\
&= \frac{1}{\sqrt{N}} \mathbb{E}[ \|h\|^{(\alpha)} ] = \frac{1}{\sqrt{N}} \mathbb{E}\left[ \max\{ \|h \|_2, \alpha \|h\|_{\infty}\} \right], %\leq \frac{1}{\sqrt{N}}\left(\E[\|h\|_2] + \alpha \E[\|h\|_{\infty}]\right),
\end{split}
\end{equation}
using the definition of fact that $\|\cdot\|_{**}^{(\alpha)} = \|\cdot\|^{(\alpha)}$, i.e., that the dual norm of $\|\cdot\|_{*}^{(\alpha)}$ is again the original norm $\|\cdot\|^{(\alpha)}$ in the second equality.

Furthermore, using Jensen's inequality, we estimate
\[
\begin{split}
\mathbb{E}\big[\|h\|_2\big] &= \frac{1}{\sqrt{N}} \mathbb{E}\Big[\| \sum_{i=1}^N \epsilon_i b_i \|_2\Big] \leq \frac{1}{\sqrt{N}} \mathbb{E}\Big[\| \sum_{i=1}^N \epsilon_i b_i \|_2^2\Big]^{1/2} \\
&= \frac{1}{\sqrt{N}}\Big(\E_{b}\big[\E_\epsilon\big[\sum_{i,k=1}^N \sum_{j=1}^m \epsilon_i \epsilon_k b_{ij}b_{kj}\big]\big]\Big)^{1/2} = \frac{1}{\sqrt{N}}\Big(\E_{b}\big[\sum_{i=1}^N \sum_{j=1}^m b_{ij}b_{ij}\big]\Big)^{1/2} \\
&= \frac{\sqrt{N m}}{\sqrt{N}} (\E|X|^2)^{1/2} = \sqrt{m},
\end{split}
\]
since $\| X\|_{L_2} = (\E|X|^2)^{1/2} = 1 $ by assumption. Assuming that $\alpha = 1$, this shows the first statement, since $\|h\|_\infty \leq \|h\|_2$ and therefore 
\[
 \mathcal{R}_N (B^{\|\cdot \|_*^{(\alpha)}}) = \frac{1}{\sqrt{N}} \mathbb{E}\big[\|h \|_2\big] \leq \sqrt{\frac{m}{N}}.
\]
To upper bound \cref{eq_R_1} for $\alpha > 1$, we calculate, using the notation $h_j = \frac{1}{\sqrt{N}} \sum_{i=1}^N \epsilon_i b_{ij}$ for the $j$-th entry of $h$, that for $q = \log (m)$
\[
\begin{split}
\E[\|h\|_{\infty}]^q &= \E\big[\max_{j=1}^m | h_j |\big]^q \leq \E\big[ \|h\|_{\ell_q}\big]^q =  \E\left[ \Big( \sum_{j=1}^m |h_j|^q\Big)^{1/q}\right]^{q} \leq
 \E\Big[\sum_{j=1}^m |h_j|^q\Big]   \\
 &=  \sum_{j=1}^m \E\big[ |h_j|^q\big] = m \E \left[\left(\frac{1}{\sqrt{N}} \sum_{i=1}^N \epsilon_i b_{i1}\right)^q \right],%%\frac{1}{\sqrt{N}}\E_{\epsilon}\big[\E_{X}[\max_{j=1}^m | h_j | | ] \big].
\end{split}
\]
where the last inequality holds since the components of $(\sum_{i=1}^N \epsilon_i b_{ij})_{j=1}^m$ are identically distributed. Since the independent random variables $(\epsilon_i b_{i1})_{i=1}^N$ are mean-zero and since they fulfill the weak moment assumption of order $\log(m)$ moments with constants $\kappa_1$ and $\gamma$, the statement of \cite[Lemma 2.8.]{ML17} implies that there exists a constant $c_0 > 0$ such that
\[
\E[\|h\|_{\infty}]^q \leq m (c_0 \mathrm{e}^{2\gamma-1} \kappa_1 \sqrt{q})^q
\]
with $q = \log(m)$ if  $ N \geq (\log(m))^{2 \gamma -1}$. Therefore
\begin{equation}
\E[\|h\|_{\infty}] \leq m^{\frac{1}{\log(m)}} c_0 \mathrm{e}^{2\gamma-1} \kappa_1 \sqrt{\log(m)} = c_0 \mathrm{e}^{2\gamma} \kappa_1 \sqrt{\log(m)},
\end{equation}
and inserting this into \cref{eq_R_1}, we obtain
\[
\begin{split}
\mathcal{R}_N (B^{\|\cdot \|_*^{(\alpha)}}) &= \frac{1}{\sqrt{N}} \mathbb{E}\left[ \max\{ \|h \|_2, \alpha \|h\|_{\infty}\} \right] \leq \frac{1}{\sqrt{N}}\left(\E[\|h\|_2] + \alpha \E[\|h\|_{\infty}]\right) \\
&\leq \frac{\sqrt{m}}{\sqrt{N}} + \e^{2\gamma}  c_0\kappa_1 \frac{\alpha \sqrt{\log(m)}}{\sqrt{N}},
\end{split}
\]
which concludes the proof of the lemma.
\end{proof}

Lemma \ref{lemma_upperbd_R} will be helpful to upper bound the term \cref{eq_complexity_para} in \cref{eq_small_ball_lemma}. To achieve a meaningful lower bound of the tail parameter $Q_S (u)$ in Lemma \ref{theoremlemma31}, we will use our next result in \Cref{lemma_lowerbound_heavytailed} providing 
a rotation invariant lower bound for super-Gaussian random vectors. 

It makes use of the following result by Montgomery-Smith about the distribution of Rademacher sums.
\begin{lemma}[{\cite{MS90}}] \label{lemma_MontgomerySmith}
There exists a constant $c > 0$ such that for all $x \in \R^m$ with $\|x\|_2=1$ and $\alpha > 0$, we have
\[
\Pro \Big(\langle \epsilon, y \rangle > \alpha \|y\|_*^{(\alpha)}\Big) \leq \e^{-\alpha^2 / 2}
\]
and 
\[
\Pro\Big(\langle \epsilon, y \rangle > \frac{\alpha}{c} \|y\|_*^{(\alpha)}\Big) \geq c^{-1} \e^{- c \alpha^2}
\]
where $\epsilon = (\epsilon_1,\ldots,\epsilon_m)$ is a vector of i.i.d.\ Rademacher variables. The constant can be chosen as $c = 4 \log(12)$.
\end{lemma}

%\subsection{Lower bound for $Q_{B^{\|\cdot \|_*^{(\alpha)}}} (u)$}
\begin{lemma} \label{lemma_lowerbound_heavytailed}
Assume that $b=(b_1,\ldots,b_m)$ is a random vector in $\R^m$ with i.i.d. entries distributed as a symmetric, unit variance random variable $X$ that is super-Gaussian with parameter $\sigma$.
%\textcolor{blue}{
%\[
%\widetilde{c} := \inf \Big\{ c \geq 0:  \inf\limits_{s > 0} \Pro \Big( |X| > \frac{1}{2} s\Big) \e^{s^2/2}  \geq c \Big\} > 0
%\]
%}
%there exists $s_0 < 10$ such that
%\[
% \Pro \Big( |X| > \frac{1}{2} s\Big)   \geq \e^{-s^2/2}
%\]
%for all $ s \geq s_0$.

Then there exist absolute constants $c_1,c_2 > 0$ such that
\[
\Pro\big( | \langle b, w\rangle | > u \big) > c_1 \e^{-\frac{c_2 u^2}{2 \sigma^2}}
\]
for all $u \geq \frac{\sigma}{4}$ and all $w \in \R^m$ such that $\|w\|_2 = 1$.
\end{lemma}
\begin{proof}
Let $t > 0$. Due to symmetry of the random variables $b_i$, $i \in [m]$, 
we can write $b_i = \epsilon'_i |b_i|$, where $\epsilon' = (\epsilon'_i)$ is a Rademacher vector independent of
$(|b_i|)$.
%in the entries of the random vector $b$ we can write the $b_i$ as $b_i = \epsilon_i^{'} |b_i|$, where the $\epsilon_i^{'}$ are Rademacher random variables that are independent of the $|b_i|$ for all $i \in [m]$. 
By conditioning on $\{|b_i|\}_{i=1}^m$, we obtain 
\begin{equation*}
\begin{split}
\Pro\big( | \langle b, w\rangle | > t \big) &= \E\left[ \Pro\Big( \Big| \sum_{i=1}^m \epsilon'_i |b_i| w_i \Big| > t  \Big\vert \{|b_i|\}_{i=1}^m\Big)  \right] \\
&= \E\left[ \Pro\Big( \Big| \sum_{i=1}^m \epsilon'_i \sign w_i |b_i| |w_i| \Big| > t  \Big\vert \{|b_i|\}_{i=1}^m\Big)  \right] \\
&= \E\left[ \Pro\Big( \Big| \sum_{i=1}^m \epsilon_i |b_i| |w_i| \Big| > t  \Big\vert \{|b_i|\}_{i=1}^m\Big)  \right],
\end{split}
\end{equation*}
where $\epsilon = (\epsilon_i)_{i \in [m]}$ is again a Rademacher vector independent of $(|b_i|)_{i \in [m]}$. 
For $r,t, f > 0$, we define the events 
\[
\mathcal{E}_{r,t,f}^{|b|} := \left\{ tf \| |b|\odot |w|\|_*^{(tf)} \geq \frac{t}{r}  \right\} \quad \mbox{ and } \quad \mathcal{E}_{r,t,f}^{|g|} := \left\{ tf\| |g|\odot |w|\|_*^{(tf)} \geq \frac{t}{r}  \right\}
\]
using the dual of the clipped norm of \cref{eq_def_dualclipped}, where $g = (g_1,\ldots,g_m)$ is a vector of standard normal i.i.d.\ entries and $ |b|\odot |w| = ( |b_i| |w_i|)_{i=1}^m$ the entrywise product of the vectors $|b|$ and $|w|$. Then
\begin{equation*}
\begin{split}
&\Pro\big( | \langle b, w\rangle | > t \big) \\
&= \E\left[ \Pro\Big( \Big| \sum_{i=1}^m \epsilon_i |b_i| |w_i| \Big| > t  \Big\vert \mathcal{E}_{r,t,f}^{|b|}, \{|b_i|\}_{i=1}^m\Big) \Pro(\mathcal{E}_{r,t,f}^{|b|} \mid \{|b_i|\}_{i=1}^m) \Big\vert \{|b_i|\}_{i=1}^m\Big)  \right]\\
&+ \E\left[ \Pro\Big( \Big| \sum_{i=1}^m \epsilon_i |b_i| |w_i| \Big| > t  \Big\vert \big(\mathcal{E}_{r,t,f}^{|b|}\big)^c ,\{|b_i|\}_{i=1}^m\Big) \Pro(\big(\mathcal{E}_{r,t,f}^{|b|}\big)^c \mid \{|b_i|\}_{i=1}^m) \Big\vert \{|b_i|\}_{i=1}^m\Big)  \right] \\
&\geq \E\left[ \E\Big[\Pro\Big( \Big| \sum_{i=1}^m \epsilon_i |b_i| |w_i| \Big| > t  \Big\vert \mathcal{E}_{r,t,f}^{|b|} , \{|b_i|\}_{i=1}^m\Big)\cdot \mathds{1}_{\mathcal{E}_{r,t,f}^{|b|}} \Big\vert \{|b_i|\}_{i=1}^m\Big]  \right] \\
&\geq \E\left[ \E\Big[\Pro\Big( \Big| \sum_{i=1}^m \epsilon_i |b_i| |w_i| \Big| > r tf \| |b|\odot |w|\|_*^{(tf)}  \Big\vert \mathcal{E}_{r,t,f}^{|b|} , \{|b_i|\}_{i=1}^m\Big)\cdot \mathds{1}_{\mathcal{E}_{r,t,f}^{|b|}} \Big\vert \{|b_i|\}_{i=1}^m\Big]  \right] 
\end{split}
\end{equation*}
It follows from \Cref{lemma_MontgomerySmith} that if we choose $r= (4 \log(12))^{-1}$, then
\begin{equation} \label{eq_lemma_lower_heavytailed_0}
\begin{split}
\Pro\big( | \langle b, w\rangle | > t \big)  &\geq r \e^{-\frac{t^2 f^2}{r}} \E\left[\E\big[\mathds{1}_{\mathcal{E}_{r,t,f}^{|b|}} \big\vert \{|b_i|\}_{i=1}^m\big]  \right]  \\
&= r \e^{-\frac{t^2 f^2}{r}} \Pro\big(  r tf \| |b|\odot |w|\|_*^{(tf)} \geq t \big).
\end{split}
\end{equation}
Since the $b_i$ are independent, symmetric, unit variance random variables that fulfill the super-Gaussian assumption \cref{eq_assumption_supergaussian_lemma} for all $i \in [m]$, we know that there exists a constant $\sigma > 0$ and independent standard normal random variables $(g_i)_{i \in [m]}$ that depend on $b_i$ for $i \in [m]$, such that $|b_i| \geq |\sigma g_i|$ almost surely.

Therefore, since $|w_i| \geq 0$ for all $i \in [m]$, it holds that
\[
\Pro\big(  r tf \| |b|\odot |w|\|_*^{(tf)} \geq t \big) \geq \Pro\big(  \sigma r tf  \| |g| \odot |w|\|_*^{(tf)} \geq t \big)
\]
with the random vector $|g| = (|g_i|)_{i \in [m]}$.
%
%
% and since we assumed that their tail has super-Gaussian behavior\footnote{Detail this more precisely?}, it is well known that there exist \textcolor{red}{standard}{\footnote{Is this true?}} normal random variables $g_i$, $i \in [m]$ that are dependent on $b_i$, but independent from each other, such that $|b_i| \geq |g_i|$ almost surely. For completeness, we show this result in the Appendix \ref{appendix}. 	Therefore
%\textcolor{red}{
%\[
%\Pro\big(  stf \| |b|\circ w\|_*^{(tf)} \geq t \big) \geq \widetilde{c} \Pro\big(  stf \| |g| \circ w\|_*^{(tf)} \geq t \big).
%\] 
%}
Next, we claim that %there exists a constant $t_0$ such that
\begin{equation} \label{eq_lemma_lower_heavytailed_1}
\Pro(\mathcal{E}_{\sigma r,t,f}^{|g|}) = \Pro\big(  \sigma r tf \| |g| \odot |w|\|_*^{(tf)} \geq t \big) \geq \e^{-\frac{125 t^2}{\sigma^2}} \quad \text{ for all } t \geq t_0 := 1/4.
\end{equation}
%and that this constant can be chose as $t_0 = 1/4$.
On the contrary, suppose that this statement does not hold, i.e., that there exists some $t \geq \frac{1}{4}$ such that 
\begin{equation} \label{eq_lemma_lower_heavytailed_2}
\Pro\big( \sigma r tf \| |g| \odot |w|\|_*^{(tf)} \geq t \big) < \e^{-\frac{125 t^2}{\sigma^2}}.
\end{equation}
Due to the first statement of \Cref{lemma_MontgomerySmith} and since the distributions of $\langle g, |w| \rangle$ and $\sum_{i=1}^m \epsilon_i |g_i| |w_i|$ coincide if $(\epsilon_i)_{i \in [m]}$ is a Rademacher vector which is independent
of $(|g_i|)_{i \in [m]}$,
%a set of Rademacher variables which are independent from each other and from the $|g_i|$, $i \in [m]$, 
we see that
\begin{equation} \label{eq_lemma_lower_heavytailed_3}
\begin{split}
&\Pro\big( | \langle g, |w| \rangle | > tf\||g| \odot |w|\|_*^{(tf)} \big\vert  (\mathcal{E}_{\sigma r,t,f}^{|g|})^{c}\big) \\
&= \E\big[\Pro\big( \Big| \sum_{i=1}^m \epsilon_i |g_i| |w_i| \Big| > tf\| |g| \odot |w| \|_*^{(tf)} \big\vert  (\mathcal{E}_{\sigma r,t,f}^{|g|})^{c} ,\{|g_i|\}_{i=1}^m\big) \big\vert (\mathcal{E}_{\sigma r,t,f}^{|g|})^{c}\big] \leq \e^{-\frac{t^2f^2}{2}}.
\end{split}
\end{equation}
Furthermore, by the conditioning on the event $(\mathcal{E}_{\sigma r,t,f}^{|g|})^{c} = \left\{ tf\| |g| w\|_*^{(tf)} < \frac{t}{\sigma r}  \right\}$, it follows that
\[
\begin{split}
&\Pro\big( | \langle g, |w| \rangle | > tf\||g| \odot |w|\|_*^{(tf)} \big\vert  (\mathcal{E}_{\sigma r,t,f}^{|g|})^{c}\big) \geq \Pro\big( | \langle g, |w| \rangle | > \frac{t}{\sigma r}\; \big\vert  (\mathcal{E}_{\sigma r,t,f}^{|g|})^{c}\big) \\
&= \Pro((\mathcal{E}_{\sigma r,t,f}^{|g|})^{c})^{-1} \Big[ \Pro\Big( | \langle g, |w| \rangle | > \frac{t}{\sigma r}\Big) - \Pro \big( \Big\{ | \langle g, |w| \rangle | > \frac{t}{\sigma r}\Big\} \cap \mathcal{E}_{\sigma r,t,f}^{|g|}  \big) \Big] \\
&\geq \Pro((\mathcal{E}_{\sigma r,t,f}^{|g|})^{c})^{-1} \Big[ \Pro\Big( | \langle g, |w| \rangle | > \frac{t}{\sigma r}\Big) - \Pro \big( \mathcal{E}_{\sigma r,t,f}^{|g|}  \big) \Big] \\
&\geq \Big[ \Pro\Big( | \langle g, |w| \rangle | > \frac{t}{\sigma r}\Big) - \Pro \big( \mathcal{E}_{\sigma r,t,f}^{|g|}  \big) \Big] \geq  \e^{-\frac{112 t^2}{\sigma^2}} - \Pro \big( \mathcal{E}_{\sigma r,t,f}^{|g|}  \big)
\end{split}
\]
for all $t \geq \frac{1}{4}$, where we use the lower bound of the Gaussian integral 
\[
\begin{split}
&\Pro\Big( | \langle g, |w| \rangle | > \frac{t}{\sigma r}\Big) = \Pro\Big( | g_1 | > \frac{t}{\sigma r}\Big) \\
&
%\geq \Pro\Big( | g_1 | > \frac{t}{\sigma r}\Big) 
\geq \sqrt{\frac{2}{\pi}}\int_{\frac{t}{\sigma r}}^{\frac{(C+1)t}{\sigma r}} \e^{-x^2/2}\mathrm{d}x 
\geq  \sqrt{\frac{2}{\pi}} \frac{C t}{\sigma r} \e^{-\frac{(C+1)^2 t^2}{2 \sigma^2 r^2}} \geq \e^{-\frac{112 t^2}{\sigma^2}}
\end{split}
\]
in the last inequality, which is true for all $t \geq \frac{1}{4}$ if $r= (4 \log(12))^{-1}$ and $C = 4 \sigma r \sqrt{\frac{\pi}{2}}$. A combination with \cref{eq_lemma_lower_heavytailed_2} and \cref{eq_lemma_lower_heavytailed_3}
yields
\[
\e^{-\frac{t^2f^2}{2}} > \e^{-\frac{112 t^2}{\sigma^2}} -  \e^{-\frac{125 t^2}{\sigma^2}}.
\]
Finally, this leads to a contradiction for $f=\frac{\sqrt{250}}{\sigma}$, $r = (4 \log(12))^{-1}$ if $t \geq \frac{\sigma}{4}$, since then the right hand side is strictly larger than $\e^{-\frac{t^2f^2}{2}}=\e^{-\frac{125 t^2}{\sigma^2}}$. This shows statement \cref{eq_lemma_lower_heavytailed_1}.

Inserting \cref{eq_lemma_lower_heavytailed_1} into \cref{eq_lemma_lower_heavytailed_0}, we see that
\[
\Pro\big( | \langle b, w\rangle | > t \big) \geq r \e^{-\frac{t^2 f^2}{r}} \e^{-\frac{125 t^2}{\sigma^2}} = r \e^{-\frac{t^2}{2}(2\cdot 250 \frac{1}{\sigma^2 r}+\frac{2\cdot 125}{\sigma^2})}
\]
for all $r \geq \sigma /4$, which concludes the assertion of the lemma with the constants $u_0 = \sigma/4$, $c_1 = \frac{1}{4 \log(12)} \approx \frac{1}{10}$ and 
$c_2 = 2000 \log(12)+ 250 \approx 5220$.
\end{proof}

%Recall that $Q_{B^{\|\cdot \|_*^{(\alpha)}}} (u) = \inf\limits_{\substack{y \in \R^m \\ \|y\|_*^{(\alpha)}=1}} \mathbb{P}(| \langle X,y \rangle | \geq u)$.
We note that the proof of the first part of the following lemma is similar to \cite[Lemma 4.3]{Litvak05}.

\begin{lemma} \label{lemma_lowerbd_Q}
Assume that $b$ is a random vector in $\R^m$ with i.i.d. entries distributed as a symmetric, unit variance random variable $X$.
\begin{enumerate}
\item If $X$ fulfills the weak moment assumption of order $4$ with constants $\gamma$ and $\kappa_1$, then  
\[
Q_{S^{\|\cdot \|_*^{(\beta)}}}\Big(\frac{1}{4}\sqrt{\frac{\log(\e N/m)}{\log(\e C)}}\Big) := \inf_{w \in S^{\|\cdot \|_*^{(\beta)}}} \mathbb{P}\Big(| \langle b,w \rangle | \geq \frac{1}{4} \sqrt{\frac{\log(\e N/m)}{\log(\e C)}}\Big) \geq 64 \sqrt{\frac{m}{N}} %Q_{S^{\|\cdot \|_*^{(\beta)}}} \Big(\frac{1}{4} \sqrt{\frac{\log(\e N/m)}{\log(\e C)}} \Big)
\]
for $\beta = \sqrt{\log\left(\frac{\e N}{m}\right)}/\sqrt{\log(\e C)}$ if $N \geq C m$, where $C:=\frac{4^{5+8 \gamma}}{3^4 \e} \kappa_1^8$.   %(old and not optimal value of the constant  C:) $C:=4^{9+8 \gamma} \kappa_1^8$. 
\item If $X$ is super-Gaussian with parameter $0 < \sigma \leq 1$,
then 
\[
 Q_{S^{\|\cdot \|_*^{(1)}}}\Big(\frac{1}{4}\sqrt{\frac{\log(\e N/m)}{\log(\e C)}}\Big) \geq 64 \sqrt{\frac{m}{N}}
\]
%\[
%Q_{S^{\|\cdot \|_*^{(1)}}}(\frac{1}{4} \sqrt{\frac{\log(\e N/m)}{\log(\e C)}}) >  \left(\frac{m}{\e N}\right)^{1/2}
%\]
if $N \geq C m$ with $C :=64^2 c_1^{-2} \e^{\frac{c_2}{16 \sigma^2}}$, where $c_1$ and $c_2$ are the constants from Lemma \ref{lemma_lowerbound_heavytailed}.
\end{enumerate}
\end{lemma}
\begin{proof}
For $\beta \geq 1$ and $u > 0$, we aim at finding a lower bound for $Q_{S^{\|\cdot \|_*^{(\beta)}}}(2u)$. Let $w \in \R^m$. It is clear from the definition of the clipped norm that $\alpha_1 \leq \alpha_2$ implies $\|w\|^{(\alpha_1)} \leq \|w\|^{(\alpha_2)}$ for all $w \in \R^m$. It follows by duality that $\alpha_1 \leq \alpha_2$ implies $\|w\|_*^{(\alpha_1)} \geq \|w\|_*^{(\alpha_2)}$. Therefore, using the norm $\| \cdot \|_{\lfloor\beta^2 \rfloor,\dagger}$ from \cref{def_dagger_norm} and the upper inequality of \Cref{lemma_norm_alpha_star}, we estimate that
\begin{equation} \label{eq_lower_bound_1}
\begin{split}
\mathbb{P}\Big(| \langle b,w \rangle | \geq 2u \|w\|_{*}^{(\beta)}\Big) &\geq \mathbb{P}\Big(| \langle b,w \rangle | \geq 2u \|w\|_{*}^{(\sqrt{\left\lfloor \beta^2 \right\rfloor})}\Big) \\
&\geq \mathbb{P}\bigg(| \langle b,w \rangle | \geq  \frac{2^{3/2}u}{\sqrt{\left\lfloor \beta^2 \right\rfloor}}  \| w \|_{\lfloor\beta^2 \rfloor,\dagger}\bigg).%  \\
%&\mathbb{P}\Big(| \langle b,e \rangle | \geq \frac{1}{4}\sqrt{\frac{\log(\e N/m)}{\log(\e C)}} \cdot \|e\|_{*}^{(\beta)}\Big) \geq \mathbb{P}\Big(| \langle b,e \rangle | \geq \frac{1}{4}\sqrt{\frac{\log(\e N/m)}{\log(\e C)}} \cdot \|e\|_{*}^{(\sqrt{\left\lfloor\frac{\log(\e N/m)}{\log(\e C)} \right\rfloor})}\Big)  \\
%& \geq \mathbb{P}\bigg(| \langle b,e \rangle | \geq \frac{\sqrt{2}}{4}\frac{\sqrt{\frac{\log(\e N/m)}{\log(\e C)}}}{\sqrt{\lfloor \frac{\log(\e N/m)}{\log(\e C)}\rfloor}} \| e \|_{(\lfloor\frac{\log(\e N/m)}{\log(\e C)} \rfloor),\dagger}\bigg) 
%\geq \mathbb{P}\bigg(| \langle b,e \rangle | \geq \frac{\sqrt{2}}{2} \| e \|_{(\lfloor\frac{\log(\e N/m)}{\log(\e C)} \rfloor),\dagger}\bigg),
\end{split}
\end{equation}
Let now $B_1,\ldots, B_{\lfloor\beta^2\rfloor}$ be a partition of $[m]$ such that
$\| w \|_{\lfloor\beta^2\rfloor,\dagger} = \sum_{\ell=1}^{\lfloor\beta^2\rfloor} \|w_{B_\ell}\|_2$,
cf.\ \cref{def_dagger_norm}. Then
\begin{equation} \label{eq_lower_bound_2}
\begin{split}
&\mathbb{P}\bigg(| \langle b,w \rangle | \geq  \frac{2^{3/2}u}{\sqrt{\left\lfloor \beta^2 \right\rfloor}}  \| w \|_{\lfloor\beta^2 \rfloor,\dagger}\bigg) \geq \mathbb{P}\bigg(| \langle b,w \rangle | \geq  \frac{2^{3/2}u}{\sqrt{\left\lfloor \beta^2 \right\rfloor}} \sum_{\ell=1}^{\lfloor\beta^2\rfloor} \|w_{B_\ell}\|_2\bigg) \\ 
%&\mathbb{P}\bigg(| \langle b,w \rangle | \geq \frac{\sqrt{2}}{2} \| w \|_{\lfloor\frac{\log(\e N/m)}{\log(\e C)}\rfloor,\dagger}\bigg) \geq \mathbb{P}\bigg(| \langle b,w \rangle | \geq \frac{\sqrt{2}}{2} \sum_{\ell=1}^{\lfloor\frac{\log(\e N/m)}{\log(\e C)}\rfloor}  \|w_{B_\ell}\|_2\bigg) \\ 
& \geq \mathbb{P}\bigg( \sum_{j \in B_\ell} b_j w_j \geq \frac{2^{3/2}u}{\sqrt{\left\lfloor \beta^2 \right\rfloor}}  \|w_{B_\ell}\|_2 \text{ for all }\ell \in \big[\lfloor \beta^2 \rfloor\big] \bigg) \\
& = \prod_{\ell =1}^{\lfloor \beta^2 \rfloor}  \mathbb{P}\Big( \sum_{j \in B_\ell} b_j w_j \geq\frac{2^{3/2}u}{\sqrt{\left\lfloor \beta^2 \right\rfloor}}  \|w_{B_\ell}\|_2\Big), %\\
%& = \prod_{\ell =1}^{\lfloor \beta^2 \rfloor} \frac{1}{2}  \mathbb{P}\Big( \Big|\sum_{j \in B_\ell} b_j w_j \Big| \geq \frac{2^{3/2}u}{\sqrt{\left\lfloor \beta^2 \right\rfloor}}  \|w_{B_\ell}\|_2\Big),
\end{split}
\end{equation}
where we used that the entries $b_1,\ldots b_m$ of $b$ are independent. Combining this with \cref{eq_lower_bound_1}, we obtain
\begin{equation} \label{eq_lower_bound_3}
\mathbb{P}\Big(| \langle b,w \rangle | \geq 2u \|w\|_{*}^{(\beta)}\Big) =  \prod_{\ell =1}^{\lfloor \beta^2 \rfloor}  \mathbb{P}\Big( \sum_{j \in B_\ell} b_j w_j \geq \frac{2^{3/2}u}{\sqrt{\left\lfloor \beta^2 \right\rfloor}}  \|w_{B_\ell}\|_2\Big).
\end{equation}

To show the first statement of the lemma, we choose $\beta =  \sqrt{\frac{\log\left(\frac{\e N}{m}\right)}{\log(\e C)}}$, where $C \geq 1$,  $u = \frac{1}{8} \sqrt{\log\left(\frac{\e N}{m}\right)}/\sqrt{\log(\e C)}$ and $N \geq C m$. By the Paley-Zygmund inequality, see, e.g.,\ \cite[Lemma 7.17]{FoucartRauhut13}, and using the symmetry of the distribution of the $b_j$, 
we find  the following lower bound for the latter probabilities,
\[
\begin{split}
&\mathbb{P}\Big( \sum_{j \in B_\ell} b_j w_j \geq\frac{2^{3/2}u}{\sqrt{\left\lfloor \beta^2 \right\rfloor}}  \|w_{B_\ell}\|_2\Big)  =  \frac{1}{2} \mathbb{P}\Big( \Big| \sum_{j \in B_\ell} b_j w_j \Big| \geq\frac{2^{3/2}u}{\sqrt{\left\lfloor \beta^2 \right\rfloor}}  \|w_{B_\ell}\|_2\Big) \\
&= \frac{1}{2} \mathbb{P}\Big( \Big| \sum_{j \in B_\ell} b_j w_j \Big| \geq \frac{\sqrt{2}}{4} \frac{\sqrt{ \frac{\log(\e N/m)}{\log(\e C)}}}{\sqrt{\left\lfloor \frac{\log(\e N/m)}{\log(\e C)} \right\rfloor}}  \|w_{B_\ell}\|_2\Big) \geq  \frac{1}{2} \mathbb{P}\Big( \Big| \sum_{j \in B_\ell} b_j w_j \Big| \geq \frac{1}{2} \|w_{B_\ell}\|_2\Big) \\
&\geq \frac{1}{2}  \frac{\big(1-(1/2)^2)\big)^2}{\|X\|_{L_4}^4} \geq \frac{ 9/16}{2 \kappa_1^4 4^{4 \gamma}} = \frac{9}{32 \kappa_1^4 4^{4 \gamma}},
\end{split}
\]
where we used that $\frac{x}{\lfloor x \rfloor} \leq 2$ for all $x > 0$ in the second inequality. Combining this lower bound with \cref{eq_lower_bound_3}, we obtain
\[
\begin{split}
&\mathbb{P}\Big(| \langle b,w \rangle | \geq \frac{1}{4}  \sqrt{\frac{\log(\e N/m)}{\log(\e C)}} \|w\|_{*}^{( \sqrt{\frac{\log(\e N /m)}{\log(\e C)}})}\Big) \geq \left(\frac{9}{32 \kappa_1^4 4^{4 \gamma}}\right)^{\lfloor\frac{\log(\e N/m)}{\log(\e C)}\rfloor} \\
&\geq \left(\frac{9}{32 \kappa_1^4 4^{4 \gamma}}\right)^{\frac{\log(\e N/m)}{\log(\e C)}} \geq \frac{1}{64} \sqrt{\frac{m}{N}}
\end{split}
\]
for $N \geq C m$ with $C = \frac{32^2}{\e^1 9^2} ( \kappa_1^4 4^{4 \gamma})^2 = \frac{4^{5+8 \gamma}}{3^4 \e} \kappa_1^8$, which concludes the first part of the lemma.  

The second statement follows from Lemma \ref{lemma_lowerbound_heavytailed} and can be considered as the special case of the above strategy for $\beta =1$: If $\beta =1$ and if $X$ fulfills the super-Gaussian assumption with parameter $\sigma$, then
\[
\begin{split}
&\mathbb{P}\Big(| \langle b,w \rangle | \geq \frac{1}{4}  \sqrt{\frac{\log(\e N/m)}{\log(\e C)}} \|w\|_{*}^{(1)}\Big)= \mathbb{P}\Big(| \langle b,w \rangle | \geq \frac{1}{4}  \sqrt{\frac{\log(\e N/m)}{\log(\e C)}} \|w\|_{2}\Big) \\
&\geq c_1 \left( \frac{\e N}{m} \right)^{-\frac{c_2}{32 \sigma^2 \log (\e C)}} \geq 64 \sqrt{\frac{m}{N}}
\end{split}
\]
if $N \geq C m$ with $C :=64^2 c_1^{-2} \e^{\frac{c_2}{16 \sigma^2}}$, where $c_1$ and $c_2$ are the constants from Lemma \ref{lemma_lowerbound_heavytailed}.
%,By \cref{eq_lower_bound_3}, we obtain then By the Paley-Zygmund inequality, e.g. \cite[Lemma 7.17]{FoucartRauhut13}, we can find a lower bound for the latter probabilities, i.e., for all $\ell \in \big[\lfloor \beta^2 \rfloor\big]$,
%\[
%\mathbb{P}\Big( \Big|\sum_{j \in B_\ell} b_j w_j \Big| \geq \frac{\sqrt{2}}{2}  \|w_{B_\ell}\|_2\Big) \geq \frac{\big(1-(\sqrt{2}/2)^2)\big)^2}{\|X\|_{L_4}^4} \geq \frac{\frac{1}{4}}{\kappa_1^4 4^{4 \gamma}} = \frac{4^{1/2}}{(\e C)^{1/4}}.
%\]
%Inserting this into \cref{eq_lower_bound_2}, we obtain
%\[
%\mathbb{P}\bigg(| \langle b,w \rangle | \geq \frac{\sqrt{2}}{2} \| w \|_{\lfloor\frac{\log(\e N/m)}{\log(\e C)}\rfloor,\dagger}\bigg) \geq \left(\frac{2}{2 (\e C)^{1/4}} \right)^{\lfloor\frac{\log(\e N/m)}{\log(\e C)}\rfloor} \geq \left(\frac{1}{\e C}\right)^{\frac{\log(\e N/m)}{4\log(\e C)}} = \left(\frac{m}{\e N}\right)^{1/4}
%\]
%and plugging this into \cref{eq_lower_bound_1}, this results in
%\[
%\mathbb{P}\Big(| \langle b,w \rangle | \geq \frac{1}{4}\sqrt{\frac{\log(\e N/m)}{\log(\e C)}} \cdot \|w\|_{*}^{(\beta)}\Big) \geq \left(\frac{m}{\e N}\right)^{1/4}
%\] 
%for all $w \in \R^m$. Furthermore, this implies that 
%\[
%Q_{B^{\|\cdot \|_*^{(\beta)}}}(2u) \geq \left(\frac{m}{\e N}\right)^{1/4}
%\]
%for $\beta = \frac{\log(\e N/m)}{\log(\e C)}$ and $u = \frac{1}{8} \sqrt{\frac{\log(\e N/m)}{\log(\e C)}}$, which is the statement of the lemma. 
%
%\textcolor{blue}{where the last inequality holds since $\frac{x}{\lfloor x \rfloor} \leq 2$ for all $x \in \R$.} 
\end{proof}

%{\textcolor{blue}{
%\begin{lemma}[Small-ball lemma 2]
%For a given $\alpha \geq 1$, define the distribution-dependent quantities
%\[
%\beta_\alpha := \inf\limits_{e \in \R^m} \mathbb{P}\Big(| \langle b_i,e \rangle | \geq \sqrt{\log\Big(\frac{eN}{m}}\Big) \|e\|_{*}^{(\alpha)}\Big)
%\]
% and $\alpha_0 := \min \left\{\alpha \geq 1 : \beta_{\alpha} > 0 \right\}$. Then
%\begin{enumerate}[label=\arabic*.]
%\item If $X$ is additionally such that there exists constants $C_1, C_2 >0 $
%\[
%\mathbb{P}\Big(| \langle b_i,e \rangle | \geq C_2  \sqrt{\log\Big(\frac{eN}{m}}\Big)\Big) \geq C_1 \cdot \left(\frac{m}{eN}\right),
%\]
%for all $e \in \{y \in \R^m : \|y\|_2 = 1 \}$,
%then
%\[
%Q_{B^{\|\cdot \|_*^{(\alpha)}}} \Big(C_2 \sqrt{\log\Big(\frac{eN}{m}}\Big)\Big) > C_1 \cdot \left(\frac{m}{eN}\right)
%\]
%%$\alpha_0 =1$, then
%\item If $X$ is such that $\alpha_0 > 1$, then
%\[
%yyy
%\]
%\end{enumerate}
%\end{lemma}
%}}

\subsection{Proof of Theorem \ref{theorem_QP}}
\begin{proof}
By \Cref{lemma_QP_equivalence}, to prove the first statement of Theorem \ref{theorem_QP}, it suffices to show \cref{eq_qp4} for the norm $\|\cdot\| = \|\cdot\|^{(\alpha)}$, $\alpha=\sqrt{\log(\e N /m)}$ with probability at least $1- 2 \exp(-2m)$. 

Consider the case that $N \geq \max\big\{Cm,\log^{2\gamma -1}(m)\big\}$ with $C=\frac{4^{5+8 \gamma}}{3^4 \e} \kappa_1^8$.
We use Lemma \ref{theoremlemma31} for the columns $a_1, \ldots, a_N$ of $A$, choosing the set $S$ such that
$S = S^{\|\cdot \|_*^{(\beta)}}$ and $\beta = \frac{\sqrt{\log\left(\frac{\e N}{m}\right)}}{\sqrt{\log(\e C)}}$, $q= \log(N)$ and $u= \frac{1}{8}\frac{\sqrt{\log\left(\frac{\e N}{m}\right)}}{\sqrt{\log(\e C)}}$.
Then, further choosing $t = \sqrt{m}$, it follows that with probability at least $1-2\e^{-2m}$,
\begin{equation} \label{eq_proof_QP_smallball}
\inf_{w \in S^{\|\cdot \|_*^{(\beta)}}} \frac{1}{N} \sum_{i=1}^N |\langle a_i, w\rangle|^q \geq \frac{\log\left(\frac{\e N}{m}\right)^{q/2}}{8^q \big(\log(\e C)\big)^{q/2}}  \left(Q_{S^{\|\cdot \|_*^{(\beta)}}} (2u) - \frac{4}{u} \mathcal{R}_N (S^{\|\cdot \|_*^{(\beta)}}) - \sqrt{\frac{m}{N}} \right)
\end{equation}
From \Cref{lemma_lowerbd_Q}, it follows that
\begin{equation}\label{eq_proof_QP_1}
Q_{S^{\|\cdot \|_*^{(\beta)}}} (2u) \geq 64 \sqrt{\frac{m}{N}}.
\end{equation} %if $N \geq \widetilde{C} m$ for $\widetilde{C}= \max\{64^4 \e, C\}$
For the complexity term $\frac{4}{u} \mathcal{R}_N (S^{\|\cdot \|_*^{(\beta)}})$, we use \Cref{lemma_upperbd_R} to see that
\begin{equation}\label{eq_proof_QP_2}
\begin{split}
\frac{4}{u} \mathcal{R}_N (S^{\|\cdot \|_*^{(\beta)}}) &\leq 32 \sqrt{\frac{m}{N}}\sqrt{\frac{\log(\e C)}{\log(\e N/m)}} + 32 \e^{2 \gamma} c_0 \kappa_1 \frac{\sqrt{\log(m)}}{\sqrt{N}} \\
&\leq (32+16)\sqrt{\frac{m}{N}} = 48 \sqrt{\frac{m}{N}},
\end{split}
\end{equation}
since $N \geq C m$ and $4 \e^{4 \gamma} c_0^2 \kappa_1^2 \log(m) \leq m$ by assumption.

Finally, inserting \cref{eq_proof_QP_1} and \cref{eq_proof_QP_2} in \cref{eq_proof_QP_smallball}, we obtain that with probability at least $1-2\e^{-2m}$,
\[
\inf_{w \in S^{\|\cdot \|_*^{(\beta)}}} \frac{1}{N} \sum_{i=1}^N |\langle a_i, w\rangle|^q \geq \frac{\left(\log(\frac{\e N}{m})\right)^{q/2}}{8^q \big(\log(\e C)\big)^{q/2}} 15 \sqrt{\frac{m}{N}}
\]
or equivalently,
\[
\begin{split}
\inf_{w \in S^{\|\cdot \|_*^{(\beta)}}} \left(\frac{1}{N} \sum_{i=1}^N |\langle a_i,w \rangle|^q\right)^{1/q}  &\geq \frac{\sqrt{\log(\frac{\e N}{m})}}{8 \sqrt{\log(\e C)}} 15^{1/\log(N)} \left(\frac{m}{N}\right)^{\frac{1}{2 \log(N)}} \\
&\geq \frac{\sqrt{\log(\frac{\e N}{m})}}{8 \sqrt{e} \sqrt{\log(\e C)}},
\end{split}
\]
since $(m/N)^{\frac{1}{2 \log (N)}} = \exp(\frac{\log(m)}{2 \log(N)}-\frac{1}{2}) > \exp(-1/2)$, which shows the assertion relative to the norm $\|\cdot \|^{(\beta)}$ and constant 
\[
D = 8 \e^{1/2} \sqrt{\log(\e C)}=8 \e^{1/2} \sqrt{1+(9+8 \gamma) \log(4) + 8 \log(\kappa_1)}.
\] 
The assertion relative to the norm $\|\cdot \|^{(\alpha)}$ for $\alpha = \sqrt{\log(\e N/m)}$ follows then trivially since $\|w\|^{(\beta)} \leq \|w\|^{(\alpha)}$, as $\beta = \frac{\sqrt{\log\left(\frac{\e N}{m}\right)}}{\sqrt{\log(\e C)}} \leq \sqrt{\log\left(\frac{\e N}{m}\right)} = \alpha$.

Similarly, we show \Cref{theorem_QP}(b). To this end, we apply Lemma \ref{theoremlemma31} for the choice $S = S^{\|\cdot \|_2} = \{w \in \R^m \mid \|w\|_2 = 1\}$, $q = \log(N)$, $u= \frac{1}{8}\frac{\sqrt{\log\left(\frac{\e N}{m}\right)}}{\sqrt{\log(\e C)}}$ and $t = \sqrt{m}$. If $N \geq C m$ with $C: = 64^2 c_1^{-2} \e^{\frac{c_2}{16 \sigma^2}}$, where $\sigma$ is the super-Gaussian parameter of $X$ and $c_1, c_2$ are the constants from \Cref{lemma_lowerbd_Q}, it follows from \Cref{lemma_upperbd_R} and \Cref{lemma_lowerbd_Q} that with probability at least $1 - 2 \e^{-2m}$,
\begin{equation} \label{eq_proof_QP_3}
\inf_{w \in S^{\|\cdot \|_{2}}} \frac{1}{N} \sum_{i=1}^N |\langle a_i, w\rangle|^q \geq  \frac{\log\left(\frac{\e N}{m}\right)^{q/2}}{8^q \big(\log(\e C)\big)^{q/2}} 31 \sqrt{\frac{m}{N}}.
\end{equation}
As above, this implies that
\[
\inf_{w \in S^{\|\cdot \|_{2}}} \left(\frac{1}{N} \sum_{i=1}^N |\langle a_i, w\rangle|^q\right)^{1/q} \geq \frac{\sqrt{\log(\frac{\e N}{m})}}{8 \e^{1/2} \sqrt{\log(\e C)}},
\]
which finishes the proof of \Cref{theorem_QP}(b) by \Cref{lemma_QP_equivalence}, i.e., $\frac{1}{\sqrt{m}}A$ fulfills the $\ell_1$-quotient property with constant 
\[
D=8 \e^{1/2} \sqrt{\log(\e C)} = 8 \e^{1/2} \sqrt{1+ 2 \log(64)+ 2 \log(c_1^{-1}) + \frac{c_2}{16 \sigma^2}}
\]
relative to $\|\cdot\|_2$ on the event of \cref{eq_proof_QP_3}.
\end{proof}
\subsection{Proof of \Cref{theorem_robustnessGuarantees}}
\begin{proof}
We first note that under the assumptions of both \Cref{theorem_robustnessGuarantees}(a) and \Cref{theorem_robustnessGuarantees}(b), it follows from \cite[Corollary V.3]{DLR16} that with probability at least $1- \exp(-c_1 m)$, $A$ fulfills the $\ell_2$-robust null space property of order $\widetilde{c}_2 s_*$ with $s_* = m/\log(\e N/m)$ and some constant $0 < \widetilde{c}_2 < 1$ and constants $\rho = 0.9$, $ \tau > 0$ relative to the $\ell_2$-norm, i.e.,
\begin{equation}\label{NSP-l2}
\|x_S\|_2 \leq \frac{\rho}{s^{1/2}}\|x_{S^c}\|_1 + \tau \|A x\|_2
\end{equation}
for all $x \in \R^m$ and $S \subset [N]$ with $ |S| = s:= \widetilde{c}_2 s_*$ and $S^c = [N] \setminus S$ if $m \geq (\log(N))^{2 \gamma -1}$, where $c_1, \widetilde{c}_2$ and $\tau$ depend on $\kappa_1$ and $\gamma$. We call the event with the latter statement $\mathcal{E}_{\text{NSP}}$.

Now we show the statement of \Cref{theorem_robustnessGuarantees}(b). Conditional on the event $\mathcal{E}_{\text{NSP}}$, it follows from \cite[Theorem II.22]{BA18} that there exist constants $C> 1$, $D> 0$ that depend only on $\kappa_1$ and $\gamma$
%and $\widetilde{c_3}$ 
such that
\[
\| x - \Delta_{1,\eta}(Ax + w)\|_{p} \leq \frac{C}{s^{1-1/p}} \sigma_s (x)_1+ s_*^{1/p-1/2}
%s_*^{1/p-1/2} 
\big(D \eta + C (\mathcal{Q}_{s_*}(A)_1 + \tau) \max\{\|w\|_2 -\eta,0\}\big)
\]
for all $1 \leq p \leq 2$ and $s \leq \widetilde{c}_2 s_*$ with $\mathcal{Q}_{s_*}(A)_1  :=\sup_{w \in \R^m\setminus\{0\}} \min_{u \in \R^N, Au=w}\frac{\sqrt{s_*}\|u\|_1}{\|w\|_2}$. By \Cref{theorem_QP}(b), it follows that there exist constants $\widetilde{C}$ and $\widetilde{D}$ such that if $N \geq \widetilde{C} m$, then with probability at least $1 - 2 \exp(-2m)$, 
\[
\mathcal{Q}_{s_*}(A)_1 \leq \widetilde{D},
\]
and we call this event $\mathcal{E}_{\text{QP-$\ell_2$}}$. The statement of \Cref{theorem_robustnessGuarantees}(b) follows then since there exists a constant $\widetilde{c_1} > 0 $ such that  $\mathcal{E}_{\text{NSP}} \cap \mathcal{E}_{\text{QP-$\ell_2$}}$ occurs with probablity at least $1-3 \exp(-\widetilde{c}_1 m)$, and on this event, the assertion follows with the constant $E:= C \widetilde{D} + \tau$.

To show Theorem \ref{theorem_robustnessGuarantees}(a), we note that in the case $\|w\|_2 \leq \eta$, the statement follows from the classical result \cite[Theorem 4.22]{FoucartRauhut13} on the event $\mathcal{E}_{\text{NSP}}$, since then there exist constants $C' > 1$, $D' > 0$ such that
\[
\| x - \Delta_{1,\eta}(Ax + w)\|_{p} \leq \frac{C'}{s^{1-1/p}} \sigma_s (x)_1+ s_*^{1/p-1/2} D' \eta
\]
for $s = \widetilde{c}_2 s_*$, for all $1 \leq p \leq 2$ and all $x \in \R^N$. Since $s \mapsto \sigma_s(x)_1/s^{1-1/p}$ is non-increasing we may replace $s$ also by a smaller value $s' < s= \widetilde{c}_2 s_*$.

Consider now the case $\|w\|_2 > \eta$ and let $z\in \R^N$ such that $\|Az - w\|_2 \leq \eta$.
%In this case, let $z \in \R^N$ be arbitrary. 
It follows again by \cite[Theorem 4.22]{FoucartRauhut13} that on the event $\mathcal{E}_{\text{NSP}}$, for $1 \leq p \leq 2$, there exist constants $C' > 1$, $D' > 0$ such that
\begin{align} %\begin{split}
& \| x - \Delta_{1,\eta}(Ax + w)\|_{p}  = \big\| (x+z)  - \Delta_{1,\eta}\big(A(x+z) + w-Az\big) - z\big\|_{p} \notag \\
&\leq  \big\| (x+z)  - \Delta_{1,\eta}\big(A(x+z) +(w-Az)\big)\big\|_p + \|z\|_{p} \notag\\
&\leq \frac{C'}{s^{1-1/p}} \sigma_s (x+z)_1 + s^{1/p-1/2} D' \eta + \|z\|_{p} \notag \\
&\leq  \frac{C'}{s^{1-1/p}} \sigma_s (x)_1 + s^{1/p-1/2} D' \eta + \frac{C'}{s^{1-1/p}} \|z\|_1+ \|z\|_{p} \notag \\
& \leq \frac{C'}{s^{1-1/p}} \sigma_s (x)_1 + s^{1/p-1/2} D' \eta + C'\big[\frac{\|z\|_1}{s^{1-1/p}} + \|z\|_{p}\big] \notag \\
& \leq \frac{C'}{s^{1-1/p}} \sigma_s (x)_1 + s^{1/p-1/2} D' \eta + C s^{1/p-1/2} \big[\frac{\|z\|_1}{s^{1/2}} + s^{1/2-1/p} \|z\|_{p}\big]. \label{eq_proof_recovery_guarantee_1}
%\end{split}
\end{align}
%if $z \in \R^N$ is such that $\|Az - w\|_2 \leq \eta$.
Moreover, if additionally $N \geq \widetilde{C} m$ and $m \geq 4 \e^{4 \gamma} c_0^2 \kappa_1^2 \log(m)$, where $\widetilde{C}$ and $c_0$ are the constants from \Cref{theorem_QP}(a), it follows from this theorem
%Theorem \ref{theorem_QP}(a) 
that with probability at least $1 - 2 \exp(-2 m)$, $A$ fulfills the $\ell_1$-quotient property relative to the norm $\|\cdot\|^{(\sqrt{\log(\e N/m)})}$ with constant $\widetilde{D}$, and we call the corresponding event $\mathcal{E}_{\text{QP-clipped}}$. Consider now on $\mathcal{E}_{\text{QP-clipped}} \cap \mathcal{E}_{\text{NSP}}$, which occurs with probability at least $1-3 \exp(-\widetilde{c}_1 m)$, a vector $\tilde{z}\in \R^N$ such that $A \tilde{z} = w$ and 
\begin{equation} \label{eq_proof_recovery_guarantee_2}
\|\tilde{z}\|_1 s_*^{-1/2} \leq \widetilde{D}  \|w\|^{(\sqrt{\log(e N/m)})},
\end{equation} which exists due to the $\ell_1$-quotient property relative to the norm $\|\cdot\|^{(\sqrt{\log(e N/m)})}$. If $z \in \R^N$ is chosen such that $z:= (1- \eta/\|w\|_2) \tilde{z}$, it holds that
\[
\|Az-w\|_2 = \big\|A\tilde{z} - w - A\tilde{z} \frac{\eta}{\|w\|_2}\big\|_2 = \frac{\eta}{\|w\|_2} \|{-A\tilde{z}}\|_2 = \eta.
\]
Choose $S$ as an index set of $s$ largest absolute coefficients of $\tilde{z}$. 
It is known \cite[Chapter 4.3]{FoucartRauhut13} that the $\ell_2$-null space property \cref{NSP-l2} 
implies the $\ell_p$-null space property for $1 \leq p \leq 2$ in the form 
\[
\|\tilde{z}_S\|_p \leq \frac{\rho}{s^{1-1/p}} \|\tilde{z}_{S^c}\|_p + \tau s^{1/p-1/2} \|A \tilde{z}\|_2.
\]
Together with Stechkin's estimate, see, e.g., \cite[Proposition 2.3]{FoucartRauhut13}, this gives
\begin{equation}  \label{eq_proof_recovery_guarantee_3}
\begin{split}
\|\tilde{z}\|_p & \leq \|\tilde{z}_S\|_p  + \|\tilde{z}_{S^c}\|_p 
\leq \frac{\rho}{s^{1-1/p}} \|\tilde{z}_{S^c}\|_1 + \tau s^{1/p-1/2} \|A \tilde{z}\|_2 + \sigma_s(\tilde{z})_p \\
&\leq \frac{\rho}{s^{1-1/p}} \|\tilde{z}\|_1 + \frac{1}{s^{1-1/p}} \|\tilde{z}\|_1 +  \tau s^{1/p-1/2} \|A \tilde{z}\|_2 
= \frac{1+\rho}{s^{1-1/p}} \|\tilde{z}\|_1 + \tau s^{1/p-1/2} \|w\|_2 \\
& \leq \frac{1+\rho}{s^{1-1/p}} \|\tilde{z}\|_1 + \tau s^{1/p - 1/2} \|w\|^{(\sqrt{\log(e N/m)})}.
\end{split}
\end{equation}
Using \cref{eq_proof_recovery_guarantee_1} and $s = \widetilde{c}_2 s_*$ we obtain
\begin{align*}
&\| x - \Delta_{1,\eta}(Ax + w)\|_{p} \\
&\leq \frac{C'}{s^{1-1/p}} \sigma_s (x)_1 + (\widetilde{c}_2 s_*)^{1/p-1/2}
\big[ D' \eta + C' \big( \frac{\|\tilde{z}\|_1}{(\widetilde{c}_2 s_*)^{1/2}} + (\widetilde{c}_2 s_*)^{1/2-1/p} \|\tilde{z}\|_{p}\big) (1- \frac{\eta}{\|w\|_2})\big] \\
&\leq \frac{C'}{s^{1-1/p}} \sigma_s (x)_1 + (\widetilde{c}_2 s_*)^{1/p-1/2}\big[ D' \eta +  C' \big(\frac{(\rho+ 2) \widetilde{D}}{\widetilde{c}_2^{1/2}} + \tau\big) \|w\|^{(\sqrt{\log(e N/m)})} (1- \frac{\eta}{\|w\|_2})\big] \\
&= \frac{C'}{s^{1-\frac{1}{p}}} \sigma_s (x)_1 + s_*^{\frac{1}{p}-\frac{1}{2}}\Big[ \frac{D'}{{\widetilde{c}_2}^{\, \frac{1}{p}-\frac{1}{2}}}  \eta +  \frac{C' \big((\rho+ 2) \widetilde{D} + {\widetilde{c}_2}^{\, 1/2} \tau\big)}{{\widetilde{c}_2}^{\, 1-\frac{1}{p}}} \frac{\|w\|^{(\sqrt{\log(e N/m)})}}{\|w\|_2} (\|w\|_2- \eta)\Big],
\end{align*}
where we used \cref{eq_proof_recovery_guarantee_2} and \cref{eq_proof_recovery_guarantee_3} in the second inequality.
%with the constants $\rho$ and $\tau$ from $\mathcal{E}_{\text{NSP}}$ and using that 
Since $s \mapsto \sigma_s(x)_1/s^{1-1/p}$ is non-increasing, again we may replace $s$ by a smaller value $s' < s= \widetilde{c_2} s_*$.
%in the first inequality and \cite[Proposition II.12.]{BA18} in the last inequality. 
This concludes the proof of \Cref{theorem_robustnessGuarantees}(a), as the constants can be defined as $C = C'$, $D = D' \widetilde{c}_2^{\, \frac{1}{2}-\frac{1}{p}}$ and $E = \widetilde{c}_2^{\,\frac{1}{p}- 1} C' \big((\rho+ 2) \widetilde{D} + {\widetilde{c}_2}^{\, 1/2} \tau\big)$.
%\textcolor{red}{To Do (Christian): Extend the following two paragraphs to cover also the case of under-estimated noise level, not only equality-constrained $\ell_1$-minimization:}
%Since $\|\cdot\|_2 \leq \|\cdot\|^{(\sqrt{\log(\e N/m)})}$
%, it follows that in this case $A$ also fulfills the $\ell_2$-robust null space property of the same order with the same constants relative to $\|\cdot\|^{(\sqrt{\log(\e N/m)})}$.
%
%Therefore, it follows from \cite[Theorem 11.12]{FoucartRauhut13} that on the intersection of these two events, which occurs with probability at least $1- 3 \e^{-\widetilde{c}_1 m}$, the desired $\ell_p$-error estimates of $\Delta_{1}$ hold for all $1 \leq p \leq 2$ with constants $C$ and $\widetilde{D}$ that depend on $\tau$, $\widetilde{c}_2$ and $D$, which proves the assertion.
%
%Theorem \ref{theorem_robustnessGuarantees}(b) combines the $\ell_2$-robust null space property relative to the norm $\|\cdot\|_2$ with the $\ell_1$-quotient property relative to $\|\cdot\|_2$.
\end{proof}

\section*{Acknowledgements}
The three authors acknowledge the support of the German-Israeli Foundation (GIF) through the grant G-1266-304.6/2015 (Analysis of structured random matrices in recovery problems).

%\section*{References}
%\bibliographystyle{abbrv}
\bibliography{Literature_QP}

\end{document}